\documentclass[a4paper,twoside,10pt]{article}

\usepackage[english]{babel}
\usepackage[utf8]{inputenc}
\usepackage{lmodern,color}
\usepackage[T1]{fontenc}
\usepackage{indentfirst}
\usepackage{amsmath,amssymb}
\usepackage[percent]{overpic}
\usepackage[a4paper,top=3cm,bottom=3cm,left=3cm,right=3cm,%
bindingoffset=5mm]{geometry}
\usepackage{lmodern}
\usepackage[T1]{fontenc}
\usepackage{lettrine}
\usepackage{booktabs}
\usepackage{subcaption}
\usepackage{authblk}
\usepackage{emp}
\usepackage{amsthm}
\usepackage{amsmath,amssymb,amsthm}
\usepackage{braket}
\usepackage{chngpage}
\usepackage{cases}
\usepackage{graphicx}
\usepackage{epstopdf}
\usepackage{tabularx}
\usepackage{multirow}
\usepackage{enumerate}
\usepackage{todonotes}
\usepackage{mathtools}
\usepackage[normalem]{ulem}
\usepackage{cancel}

\newcommand{\numberset}{\mathbb}
\newcommand{\N}{\numberset{N}}
\newcommand{\R}{\numberset{R}}

\newcommand{\Pk}{\numberset{P}}

\renewcommand{\epsilon}{\varepsilon}
\renewcommand{\theta}{\vartheta}
\renewcommand{\rho}{\varrho}
\renewcommand{\phi}{\varphi}

\newcommand{\xx}{\boldsymbol{x}}

\newcommand{\vv}{\boldsymbol{v}}

\newcommand{\nn}{\boldsymbol{n}}

\newcommand{\aaa}{\boldsymbol{\alpha}}
\newcommand{\bb}{\boldsymbol{\beta}}

\newcommand{\gr}{\nabla}

\newcommand{\eek}{\mathcal{E}^{E,k}}
\newcommand{\pek}{\mathcal{P}^{E,k}}

\newcommand{\ek}{\mathcal{E}^{k}}
\newcommand{\pk}{\mathcal{P}^{k}}
\newcommand{\nablah}{\nabla_{\!h}}

\def\PN{{\Pi^{\nabla, E}_{k}}}
\def\P0{{\Pi^{0, E}_{k}}}
\def\Pg{{\Pi^{0}_{k}}}
\def\PZ0P{{\boldsymbol{\Pi}^{0, E}_{k-1}}}
\def\PP0P{{\boldsymbol{\Pi}^{0, E}_{k}}}

\def\Stab{\mathcal{S}^E}
\def\AcipE{\mathcal{A}_{\rm cip}^E}

\def\Acip{\mathcal{A}_{\rm cip}}

\def\Atcip{\widetilde{\mathcal{A}}_{\rm cip}}

\def\vint{v_{\mathcal{I}}}
\def\uint{u_{\mathcal{I}}}
\def\eint{e_{\mathcal{I}}}

\def\cip{{\rm cip}}
\def\cipdual{{\rm cip^*}}
\def\errF{\eta_{\mathcal{F}}}

\def\errBt{\eta_{\tilde{\mathcal{B}}}}
\def\erra{\eta_{a}}
\def\pwh{  h\pi ( \bb_h \cdot \nablah \Pg v_h )}
\def\wh{  h \bb_h \cdot \nabla \Pg v_h}

\def\errb{\eta_{b}}
\def\errc{\eta_{c}}
\def\errJ{\eta_{J}}
\def\errN{\eta_{\mathcal N}}

\def\reg{k}

{\left\lbrace\begin{array}{@{}l@{}}}%
{\end{array}\right.}
\theoremstyle{definition}

\theoremstyle{remark}
\newtheorem{remark}{Remark}[section]

\theoremstyle{remark}

\theoremstyle{plain}
\newtheorem{theorem}{Theorem}[section]
\newtheorem{proposition}{Proposition}[section]

\newtheorem{lemma}{Lemma}[section]

\usepackage{lipsum}

\usepackage{authblk}
\usepackage{color}

\usepackage{setspace}

\title{\textbf{A Nonconforming Virtual Element Method for Advection-Diffusion-Reaction Problems with CIP Stabilization}}

\date{\today}

\author[1,4]{C. Lovadina \thanks{carlo.lovadina@unimi.it}}

\author[2]{I. Perugia \thanks{ilaria.perugia@univie.ac.at}}

\author[3]{M. Trezzi \thanks{manuelluigi.trezzi01@universitadipavia.it}}

\affil[1]{Dipartimento di Matematica ``F. Enriques'',
Universit\`a degli Studi di Milano,
Via Cesare Saldini 50 - 20133 Milano, Italy}

\affil[2]{Faculty of Mathematics,
University of Vienna,
Oskar-Morgenstern-Platz 1 - 1090 Vienna, Austria}

\affil[3]{Dipartimento di Matematica ``F. Casorati'',
Universit\`a degli Studi di Pavia,
Via Adolfo Ferrata 5 - 27100 Pavia, Italy}   

\affil[4]{IMATI-CNR, Via Adolfo Ferrata 5 - 27100 Pavia, Italy}

\begin{document}

\maketitle

\begin{abstract}
We study a nonconforming virtual element method (VEM) for advection-diffusion-reaction problems with continuous interior penalty (CIP) stabilization. The design of the method is based on a standard variational formulation of the problem (no skew-symmetrization), and boundary conditions are imposed with a Nitsche technique. 
We use the enhanced version of VEM, with a ``dofi-dofi''  stabilization in the diffusion and reaction terms. 
We prove stability of the proposed method and derive $h$-version error estimates.
\end{abstract}

\section{Introduction}

The Virtual Element Method (VEM) is a paradigm for the discretization of problems in partial differential equations \cite{volley, hitchhikers}, that is experiencing significant success in a quite large range of applications; 
see for instance \cite{sema-simai} and \cite{ACTA-VEM} and the references therein.
In this paper, we consider the VEM to approximate the solution to diffusion-reaction-advection scalar problems. As it is well-known, these are
elliptic problems which, however, displays severe numerical difficulties when the diffusive term, carrying ellipticity, becomes ``small'' with respect to the others.
In particular, standard Galerkin schemes, such as the basic version of the Finite Element Method (FEM), typically exhibit unphysical 
oscillations of the discrete solution in the advection-dominated regime, unless extremely fine meshes are employed. 
In addition, it is worth noting that a similar situation occurs in more complex fluid-dynamics problems, e.g. the Oseen and the Navier-Stokes equation. 
Thus, the scalar diffusion-reaction-advection problem may represent also a useful playground for the development of numerical methods for the above-mentioned applications.
Many stabilizing techniques have already been developed in the framework of the FEM. 
To mention only a few of them, we recall upwind Discontinuous Galerkin schemes \cite{DE-book,R-book,BrezziDG}, Streamline Upwind Petrov-Galerkin (SUPG) and variants \cite{brooks1982}, Continuous Interior Penalty (CIP) \cite{douglas, burman:2004, burman:2007}, and Local Projection Stabilization (LPS)\cite{LPS0,matthies2007unified}.
All these approaches aim to offer a quasi-robust method with respect to the diffusion parameter.
From the theoretical viewpoint, a method is quasi-robust if, assuming sufficiently regular solutions and data, it is possible to derive error estimates that are uniform in the diffusion parameter, with respect to a norm that gives a suitable control on the convective term.
Nowadays, the last above-mentioned three methodologies have their VEM counterparts: for instance, see \cite{berrone:2016,BERRONE2018,BDLV:2021} for SUPG,  \cite{li2021} for LPS and \cite{BLT:2024} for CIP (regarding other polygonal technologies, one may refer to \cite{HHO-book-1,HHO-book-2}). 

This contribution aims to develop, analyse, and numerically test CIP methods using \emph{nonconforming} VEMs, see \cite{AML-2016}. Hence, we expand upon
the results of \cite{BLT:2024}, which focuses on the $H^1$-conforming case instead. 

A reason for our interest in the nonconforming setting is that the design, implementation and analysis of nonconforming VEMs are independent of the spatial dimension.
Furthermore, we highlight that nonconforming methods are used to avoid locking phenomena in the simulation of incompressible fluids (e.g. for FEM: the Crouzeix-Raviart element). 
In this regard, as previously noted,
our present contribution can be viewed
as a first step towards the design of quasi-robust methods for incompressible
fluid problems.

Our nonconforming methods exhibit the following distinctive
features compared to~\cite{BLT:2024}, which also entail deviations within the stability and error analysis guidelines developed there.
\begin{itemize}
\item We discretize the standard variational formulation of the problem: we do not need to skew-symmetrize the discrete advection form.   
\item We use
Nitsche's method to impose Dirichlet boundary conditions as in \cite{BLT:2024}, but we propose a symmetric version of it, which is robust in the vanishing advection limit. 
\end{itemize}

Indeed, in such a situation, only the symmetric version of the Nitsche method, which lead to an adjoint consistent discrete scheme, allows to maintain optimal $L^2$-error estimates for regular problems. We also remark that, to the best of our knowledge, for nonconforming VEMs no analysis of the Nitsche method has been developed yet, not even for the standard Laplace equation.

\medskip

The paper is organized as follows. After presenting the continuous and discrete problems in Section \ref{s:VEM}, we develop 
stability and convergence analysis of the CIP-stabilizad nonconforming VEM in Section \ref{s:theory}. Finally, numerical tests are shown in Section \ref{s:numerical}.
Throughout the paper, we use standard notation for Sobolev norms and semi-norms. 

\section{The Continuous and the Discrete Problem}\label{s:VEM}

\subsection{Continuous Problem}\label{sec:continuous_problem}

Given a polytopal domain $\Omega \subset \R^d$, $d=2,3$, with boundary $\Gamma$, we consider the steady advection-diffusion-reaction equation with non-homogeneous Dirichlet boundary condition
\begin{equation}\label{eq:strong}
\left\{
\begin{aligned} 
- \epsilon \Delta u + \bb \cdot \nabla u + \sigma u &= f \quad \text{in }\Omega \, , \\
u &= g \quad \text{on } \Gamma \, .
\end{aligned}
\right.
\end{equation}
For simplicity, we assume that the diffusion and the reaction coefficient $\epsilon$ and $\sigma$, respectively, are two positive constants, while the advection coefficient $\bb \in [W^{1,\infty}(\Omega)]^d$ is such that $\text{div}(\bb) = 0$.
Finally, $f\in L^2(\Omega)$ and $g \in H^{\frac{1}{2}}(\Gamma)$. 
We introduce the Sobolev space
\[
V_g \coloneqq \{ v \in H^1(\Omega) \, \text{ s.t. } v|_\Gamma = g \} \, ,
\]
where the restriction to the boundary is intended in the sense of the trace operator.
We assume that the boundary $\Gamma$ is split into two disjoint parts
\begin{equation*}
\Gamma_{\text{in}} \coloneqq \{ \xx \in \Gamma \, | \, (\bb(\xx) \cdot \nn) < 0 \}
\quad \text{and} \quad 
\Gamma_{\text{out}} \coloneqq \{ \xx \in \Gamma \, | \, (\bb(\xx) \cdot \nn) \geq 0 \} \, ,
\end{equation*}
where $\nn$ is the outward unit normal vector to the boundary~$\Gamma$.  
Integrating by parts the first term in~\eqref{eq:strong}, we derive a weak formulation for problem~\eqref{eq:strong}:
\begin{equation}\label{eq:weak-2}
\left\{
\begin{aligned}
&\text{find }u \in V_g \text{ such that:} \\
&\epsilon \, a(u,v) + b(u,v) + \sigma \, c(u,v) = \mathcal{F}(v) \quad \forall v \in V \coloneqq H_0^1(\Omega) \, .
\end{aligned}
\right .
\end{equation}
The forms
$a(\cdot,  \cdot) \colon V_g \times V \to \R$ , 
$b(\cdot,  \cdot) \colon V_g \times V \to \R$
and
$c(\cdot,  \cdot) \colon V_g \times V \to \R$
are defined as
\begin{equation}
\label{eq:a-c}
a(u,  v) \coloneqq \int_{\Omega} \nabla u \cdot \nabla v \, {\rm d}\Omega
\qquad \text{for all $u \in V_g \, , v \in V$,}
\end{equation}
\begin{equation}
\label{eq:b-c}
b(u,  v) := \int_{\Omega} (\bb \cdot \nabla u) \, v \, {\rm d}\Omega 
\qquad \text{for all $u \in V_g \, , v \in V$,}
\end{equation}
\begin{equation}
\label{eq:c-c}
c(u,  v) := \int_{\Omega}  u \, v \, {\rm d}\Omega
\qquad \text{for all $u \in V_g \, , v \in V$.}
\end{equation} 
Finally, the linear operator $\mathcal{F} (\cdot) \colon V \to \R$ is defined as
\begin{equation}\label{eq:f-c}
\mathcal{F}(v) \coloneqq \int_\Omega f \, v \, {\rm d}\Omega 
\, .
\end{equation}
It is well know that standard Galerkin schemes for discretizing~\eqref{eq:weak-2} could produce unsatisfactory solutions. 
In particular, if the mesh size is not sufficiently small, the numerical solutions may exhibit spurious oscillations.
One possibility to overcome this difficulty is to insert in the weak formulation~\eqref{eq:weak-2} a term that penalizes the jump of the gradient of the functions.
This idea was originally proposed in \cite{douglas} and we will explain the details in the following sections. 
From now on, we assume that the parameters in equation~\eqref{eq:strong} are scaled so that 
\begin{equation}\label{eq:beta-scaling}
\| \bb \|_{[L^\infty(\Omega)]^d} = 1 \, . 
\end{equation}
Since we are mainly interested in the advection-dominated case, we also assume that the diffusion coefficient $\epsilon$ is such that $\epsilon \ll 1$.
In the estimates in the rest of the paper, we keep track of the dependence on $\epsilon$ and~$h$. Thus, we will use the notation $\lesssim\cdot$ for $\le C\,\cdot\,$, with hidden constant $C>0$ independent of~$h$ and~$\epsilon$.

\subsection{Notation and Polynomial Approximation Results} \label{ss:preliminary}

We consider a sequence $\{\Omega_h\}_h$ of tessellations of the domain~$\Omega$ into non-overlapping polytopes. 
From now on, let~$E$ denote an (open) element of the mesh~$\Omega_h$, and~$e$ facet (edge if $d=2$, face if $d=3$) of the element $E$. Moreover, let $|E|$, $h_E$, and $\nn^E$ denote the volume, the diameter, and the unit outward normal of the element~$E$, respectively. We also write $\nn^e$ for the restriction of $\nn^E$ to the facet $e\subset \partial E$.
We define the mesh size as $h \coloneqq \sup_{E\in\Omega_h}h_E$. 

The set of the facets of a tessellation $\Omega_h$ is denoted with $\mathcal{E}_h$.
This set is split in internal facets and boundary facets $\mathcal{E}_h = \mathcal{E}_h^0 \cup \mathcal{E}_h^\partial$. For any scalar-valued function $v$, whose restriction $v^E$ to each element $E\in\Omega_h$ is sufficiently smooth, we define the (vector-valued) jump $[\![v]\!]$ and the (scalar-valued) average $\{v\}$ on each facet $e\in \mathcal{E}_h^0$, $e\subset\partial E\cap \partial E'$ as
\begin{equation}\label{eq:jmp-avrg}
[\![v]\!] \coloneqq v^E \nn^E + v^{E'} \nn^{E'},\qquad
\{v\} \coloneqq \dfrac{v^{E} + v^{E'}}{2}\,.
\end{equation}
Similarly, for any piecewise smooth vector-valued function $\vv$, we define the (scalar-valued) jump~$[\![\vv]\!]$ on each facet $e\in \mathcal{E}_h^0$, $e\subset\partial E\cap \partial E'$ as
\begin{equation}\label{eq:vect-jmp}
[\![\vv]\!] \coloneqq \vv^E\cdot \nn^E + \vv^{E'}\cdot \nn^{E'}.
\end{equation}
We also introduce the notation $\nablah$ for the elementwise gradient operator.

We suppose that $\set{\Omega_h}_h$ fulfils the following assumption.
\medskip

\noindent
\textbf{(A1) Mesh assumption.}
There exists a positive constant $\rho$ such that, for any $E \in \set{\Omega_h}_h$,   
\begin{itemize}
\item $E$ is star-shaped with respect to a ball $B_E$ of radius $ \geq\, \rho \, h_E$;
\item any facet $e$ of $E$ has diameter $h_e$ $ \geq\, \rho \, h_E$;
\item for $d=3$, every face $e$ is star-shaped with respect to a ball $B_e$ of radius $ \geq\, \rho \, h_e$;
\item the mesh is quasi-uniform, i.e., any element has diameter $h_E \geq \rho h$.
\end{itemize} 

We introduce some basic spaces that are useful below. Given two positive integers $n$ and $m$, and $p\in [0,+\infty]$, for any $E \in \Omega_h$ we define
\begin{itemize}
\item $\Pk_n(E)$: the set of polynomials on $E$ of degree $\leq n$  (with $\Pk_{-1}(E)=\{ 0 \}$),
\item $\Pk_n(e)$: the set of polynomials on $e$ of degree $\leq n$  (with $\Pk_{-1}(e)=\{ 0 \}$),
\item $\Pk_n(\Omega_h) := \{q \in L^2(\Omega) \quad \text{s.t} \quad q|_E \in  \Pk_n(E) \quad \text{for all $E \in \Omega_h$}\}$,
\item $W^{m,p}(\Omega_h) := \{v \in L^2(\Omega) \quad \text{s.t} \quad v|_E \in  W^{m,p}(E) \quad \text{for all $E \in \Omega_h$}\}$ equipped with the broken norm and seminorm
\[
\begin{aligned}
&\|v\|^p_{W^{m,p}(\Omega_h)} := \sum_{E \in \Omega_h} \|v\|^p_{W^{m,p}(E)}\,,
\qquad 
&|v|^p_{W^{m,p}(\Omega_h)} := \sum_{E \in \Omega_h} |v|^p_{W^{m,p}(E)}\,,
\qquad & \text{if $1 \leq p < \infty$,}
\\
&\|v\|_{W^{m,p}(\Omega_h)} := \max_{E \in \Omega_h} \|v\|_{W^{m,p}(E)}\,,
\qquad 
&|v|_{W^{m,p}(\Omega_h)} := \max_{E \in \Omega_h} |v|_{W^{m,p}(E)}\,,
\qquad & \text{if $p = \infty$,}
\end{aligned}
\]
\end{itemize}
We introduce the following polynomial projections: 
\begin{itemize}
\item the $\boldsymbol{L^2}$\textbf{-projection} $\Pi_n^{0, E} \colon L^2(E) \to \Pk_n(E)$, given by
\begin{equation*}
\label{eq:P0_k^E}
\int_Eq_n (v - \, {\Pi}_{n}^{0, E}  v) \, {\rm d} E = 0 \qquad  \text{for all $v \in L^2(E)$  and $q_n \in \Pk_n(E)$,} 
\end{equation*} 
with obvious extensions for functions defined on a facet $\Pi_n^{0, e} \colon L^2(e) \to \Pk_n(e)$, and for vector-valued functions $\boldsymbol{\Pi}^{0, E}_{n} \colon [L^2(E)]^2 \to [\Pk_n(E)]^2$ and $\boldsymbol{\Pi}^{0, e}_{n} \colon [L^2(e)]^2 \to [\Pk_n(e)]^2$;

\item the $\boldsymbol{H^1}$\textbf{-seminorm projection} ${\Pi}_{n}^{\nabla,E} \colon H^1(E) \to \Pk_n(E)$, defined by 
\begin{equation*}
\label{eq:Pn_k^E}
\left\{
\begin{aligned}
& \int_E \gr  \,q_n \cdot \gr ( v - \, {\Pi}_{n}^{\nabla,E}   v)\, {\rm d} E = 0 \quad  \text{for all $v \in H^1(E)$ and  $q_n \in \Pk_n(E)$,} \\
& \int_{\partial E}(v - \,  {\Pi}_{n}^{\nabla, E}  v) \, {\rm d}s= 0 \, .
\end{aligned}
\right.
\end{equation*}
\end{itemize}
The global counterparts of these operators, $\Pi_n^{0}$ and $\Pi_n^{\nabla}$, are defined, respectively, as
\begin{equation*}
\label{eq:proj-global}
(\Pi_n^{0} v)|_E = \Pi_n^{0,E} v \,,
\qquad
(\Pi_n^{\nabla} v)|_E = \Pi_n^{\nabla,E} v 
\qquad
\text{for all $E \in \Omega_h$\,.}
\end{equation*}
We recall a classical result for polynomials on star-shaped domains (see, for instance,~\cite{brenner-scott:book}).

\begin{lemma}[Polynomial approximation]
\label{lm:bramble}
Under assumption \textbf{(A1)}, for any $E \in \Omega_h$ and for any sufficiently smooth function $\phi$ defined on $E$, we have that 
\[
\begin{aligned}
&\|\phi - \Pi^{0,E}_n \phi\|_{W^{m,p}(E)} \lesssim h_E^{s-m} |\phi|_{W^{s,p}(E)} 
\qquad & \text{$s,m \in \N$, $m \leq s \leq n+1$, $p=1, \dots, \infty$,}
\\
&\|\phi - \Pi^{\nabla,E}_n \phi\|_{m,E} \lesssim h_E^{s-m} |\phi|_{s,E} 
\qquad & \text{$s,m \in \N$, $m \leq s \leq n+1$, $s \geq 1$,}
\\
&\|\nabla \phi - \boldsymbol{\Pi}^{0,E}_{n} \nabla \phi\|_{m,E} \lesssim h_E^{s-1-m} |\phi|_{s,E} 
\qquad & \text{$s,m \in \N$, $m+1 \leq s \leq n+2$. 
}
\end{aligned}
\]
\end{lemma}

We also have a counterpart of this result for projections defined on a facet.
\begin{lemma}\label{lm:bh-edge}
Let $e \in \mathcal{E}_h^0$ be an internal facet and let $E, E' \in \Omega_h$ be such that $e \subset \partial E \cap \partial E'$. Moreover, let $\Pi^{0,e}_n$ denote the $L^2(e)$-orthogonal projection onto $\Pk_{n}(e)$. Then, for every sufficiently smooth function $\varphi$ defined on $E \cup E'$, we have that
\begin{equation}\label{eq:edge_est}
\| \varphi - \Pi^{0,e}_n \varphi \|_{0,e} 
\lesssim
h_E^{m-\frac{1}{2}} \bigl( | \varphi |_{m, E } + | \varphi |_{m,  E'}\bigr) \qquad m,n\in\N, \, 1 \leq m \leq n+1 \, .
\end{equation}

\end{lemma}

\begin{proof}
Considering the element $E$, for $m,n\in\N$, with $1 \leq m \leq n+1$, we have 
\begin{equation*}
\begin{aligned}
\| \varphi - \Pi^{0,e}_n \varphi \|_{0,e} 
\leq\| \varphi - \Pi^{0,E}_n \varphi \|_{0,e}
&\lesssim \bigl(  h_E^{-1/2} \|\varphi - \Pi^{0,E}_n \varphi  \|_{0,E} + h_E^{1/2}|\varphi - \Pi^{0,E}_n \varphi  |_{1,E}\bigr)\\
&\lesssim h_E^{m-\frac{1}{2}} | \varphi |_{m, E } .
\end{aligned}
\end{equation*}
Since the same estimate applies also to the element $E'$, bound \eqref{eq:edge_est} follows.  
\end{proof}

\subsection{Virtual Element Spaces}\label{ss:vemspaces}

Given an element $E \in \Omega_h$ with $n_E$ facets, and a positive integer $k$, we define the space
\begin{equation}\label{eq:local-space}
\begin{aligned}
V^{k,nc}_h(E) \coloneqq
\bigl\{
v_h \in &H^1(E) \quad \text{s.t.} \quad 
\nabla v_h\cdot \nn^E
\in \Pk_{k-1}(e) \quad \text{for all $e \subset \partial E$,} 
\bigr .
\\
\bigl .
& \Delta v_h \in \Pk_k(E) \,, \quad
(v_h - \PN v_h, \, \widehat{p}_k ) = 0 \quad 
\text{for all $\widehat{p}_k \in \Pk_k(E) / \Pk_{k-2}(E)$}
\bigr \} \,.
\end{aligned}
\end{equation}
Contrary to the classical virtual element space presented in \cite{volley}, the space $V^{k,nc}_h(E)$ contains functions that are solutions of a PDE with Neumann boundary conditions. 
The space \eqref{eq:local-space} is the enhanced version of the space originally proposed in \cite{AML-2016}; more details can be found in \cite{CMS-2018}.
The following set of linear operators represents a set of degrees of freedom (DoFs) for the space $V^{k,nc}_h(E)$:
\begin{itemize}
    \item $\eek$: the moments up to the order $k-1$ on each facet $e\subset\partial E$:
    \begin{equation}\label{eq:dofsEdges}
    \mu_{E,e}^{{\boldsymbol \ell}}(v_h) \coloneqq \dfrac{1}{|e|} \int_e v_h \left( \dfrac{{\bf s} - {\bf s}_e}{h_e} \right)^{{\boldsymbol \ell}} {\rm d}s \, , \quad 
    |{\boldsymbol \ell}|\le k-1,
    \end{equation}
    where ${\bf s}$ is expressed in the local~$d-1$ coordinates on~$e$, ${\bf s}_e$ is the barycenter of $e$, and~${\boldsymbol \ell}\in\N^{d-1}$ is a multi-index with $d-1$ components; 
    \item $\pek$: the moments up to the order $k-2$ on $E$:
    \begin{equation}\label{eq:dofsInt}
    \mu^{\aaa}_E(v_h) \coloneqq \dfrac{1}{|E|} \int_E v_h \left( \dfrac{\xx - \xx_E}{h_E} \right)^{\aaa} {\rm d}E \, , \quad |\aaa| \leq k-2, 
    \end{equation}
    where $\xx_E$ is the barycenter of $E$, and $\aaa\in\N^d$ is a multi-index with $d$ components.
\end{itemize} 
In particular, for $k=1$, the basis functions for the space $V_h^{k,nc}(E)$ are associated only with the facet DoFs; for $k\geq 2$, there are also the interior DoFs. In Figure \ref{fig:dofs}, we show the DoFs on a penthagon for $k=1,2,3$.

\begin{figure}
\begin{minipage}{0.31\textwidth}
\centering
\begin{tikzpicture}[scale = 0.9]
\coordinate (A) at (90:1.8);
\coordinate (B) at (162:1.8);
\coordinate (C) at (234:1.8);
\coordinate (D) at (306:1.8);
\coordinate (E) at (378:1.8);

\path (A) +(216:1.058) coordinate (A1);
\path (B) +(288:1.058) coordinate (B1);
\path (C) +(0:1.058) coordinate (C1);
\path (D) +(72:1.058) coordinate (D1);
\path (E) +(144:1.058) coordinate (E1);

\coordinate (01) at (0,0);

\draw (A) -- (B) -- (C) -- (D) -- (E) -- (A);

\draw[red, fill] (A1) circle (3pt);
\draw[red, fill] (B1) circle (3pt);
\draw[red, fill] (C1) circle (3pt);
\draw[red, fill] (D1) circle (3pt);
\draw[red, fill] (E1) circle (3pt);
\end{tikzpicture}
\begin{center}
$k=1$
\end{center}
\end{minipage}
\hfill
\begin{minipage}{0.31\textwidth}
\centering
\begin{tikzpicture}[scale = 0.9]
\coordinate (A) at (90:1.8);
\coordinate (B) at (162:1.8);
\coordinate (C) at (234:1.8);
\coordinate (D) at (306:1.8);
\coordinate (E) at (378:1.8);

\path (A) +(216:0.705) coordinate (A1);
\path (B) +(288:0.705) coordinate (B1);
\path (C) +(0:0.705) coordinate (C1);
\path (D) +(72:0.705) coordinate (D1);
\path (E) +(144:0.705) coordinate (E1);

\path (A1) +(216:0.705) coordinate (A2);
\path (B1) +(288:0.705) coordinate (B2);
\path (C1) +(0:0.705) coordinate (C2);
\path (D1) +(72:0.705) coordinate (D2);
\path (E1) +(144:0.705) coordinate (E2);

\coordinate (O1) at (90:0.0);

\draw (A) -- (B) -- (C) -- (D) -- (E) -- (A);

\draw[red, fill] (A1) circle (3pt);
\draw[red, fill] (B1) circle (3pt);
\draw[red, fill] (C1) circle (3pt);
\draw[red, fill] (D1) circle (3pt);
\draw[red, fill] (E1) circle (3pt);

\draw[red, fill] (A2) circle (3pt);
\draw[red, fill] (B2) circle (3pt);
\draw[red, fill] (C2) circle (3pt);
\draw[red, fill] (D2) circle (3pt);
\draw[red, fill] (E2) circle (3pt);

\draw[blue, fill] (O1) circle (3pt);
\end{tikzpicture}
\begin{center}
$k=2$
\end{center}
\end{minipage}
\hfill
\begin{minipage}{0.31\textwidth}
\centering
\begin{tikzpicture}[scale = 0.9]
\coordinate (A) at (90:1.8);
\coordinate (B) at (162:1.8);
\coordinate (C) at (234:1.8);
\coordinate (D) at (306:1.8);
\coordinate (E) at (378:1.8);

\path (A) +(216:0.54) coordinate (A1);
\path (B) +(288:0.54) coordinate (B1);
\path (C) +(0:0.54) coordinate (C1);
\path (D) +(72:0.54) coordinate (D1);
\path (E) +(144:0.54) coordinate (E1);

\path (A1) +(216:0.54) coordinate (A2);
\path (B1) +(288:0.54) coordinate (B2);
\path (C1) +(0:0.54) coordinate (C2);
\path (D1) +(72:0.54) coordinate (D2);
\path (E1) +(144:0.54) coordinate (E2);

\path (A2) +(216:0.54) coordinate (A3);
\path (B2) +(288:0.54) coordinate (B3);
\path (C2) +(0:0.54) coordinate (C3);
\path (D2) +(72:0.54) coordinate (D3);
\path (E2) +(144:0.54) coordinate (E3);

\coordinate (O1) at (90:0.3);
\coordinate (O2) at (210:0.3);
\coordinate (O3) at (330:0.3);

\draw (A) -- (B) -- (C) -- (D) -- (E) -- (A);

\draw[red, fill] (A1) circle (3pt);
\draw[red, fill] (B1) circle (3pt);
\draw[red, fill] (C1) circle (3pt);
\draw[red, fill] (D1) circle (3pt);
\draw[red, fill] (E1) circle (3pt);

\draw[red, fill] (A2) circle (3pt);
\draw[red, fill] (B2) circle (3pt);
\draw[red, fill] (C2) circle (3pt);
\draw[red, fill] (D2) circle (3pt);
\draw[red, fill] (E2) circle (3pt);
\draw[red, fill] (A3) circle (3pt);
\draw[red, fill] (B3) circle (3pt);
\draw[red, fill] (C3) circle (3pt);
\draw[red, fill] (D3) circle (3pt);
\draw[red, fill] (E3) circle (3pt);

\draw[blue, fill] (O1) circle (3pt);
\draw[blue, fill] (O2) circle (3pt);
\draw[blue, fill] (O3) circle (3pt);
\end{tikzpicture}
\begin{center}
$k=3$
\end{center}
\end{minipage}
\caption{Degrees of freedom for a penthagon.}\label{fig:dofs}
\end{figure}
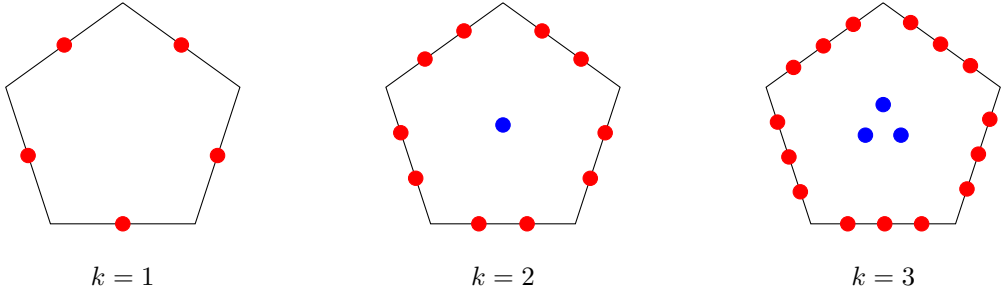

The dimension $N_E$ of the space $V_h^{k,nc}(E)$ is 
\[
N_E = \begin{cases}
k \, n_E + \dfrac{(k-1)(k-2)}{2} \qquad&\text{if $d=2$}\, ,\\[0.3cm]
\dfrac{(k+1)k}{2}n_E+\dfrac{(k+1)k(k-1)}{6}\qquad&\text{if $d=3$}\, .
\end{cases}
\]

\begin{remark}\label{rem:proj}
For any $v_h\in V_h^{k,nc}(E)$, the projections $\Pi_k^{0, E}v_h$, $\PN v_h$, and  $\PZ0P\nabla v_h$
are computable in terms of the DoFs of $v_h$; see for instance \cite{AML-2016,CMS-2018}.
\end{remark}

For every mesh $\Omega_h$, we introduce the global virtual element space as
\begin{equation}
\begin{aligned}
V^{k,nc}_h(\Omega_h) \coloneqq
\bigl\{
v_h \in W^{1,2}(\Omega_h) \quad \text{s.t.} \quad &
v_h|_E \in V_h^{k,nc}(E) \quad \forall E \in \Omega_h\, , \\ 
&\int_e [\![v_h]\!]\cdot\nn^e\, 
q \, {\rm d}s= 0 \quad  \forall q \in \mathbb P_{k-1}(e) \quad \forall e \in {\mathcal E}_h^0
\bigr\} \, .
\end{aligned}
\end{equation}
We remark that we do not impose full continuity across the element boundaries. 
On each interior facet, we only require that the moments up to order $k-1$ are preserved.
Therefore, $V_h^{k,nc}(\Omega_h) \not\subset H^1(\Omega)$.
The global DoFs are as follows:
\begin{itemize}
    \item $\ek$: the moments up to the order $k-1$ on each $e \in \mathcal E$ :
    \[
    \mu_{e}^{\boldsymbol \ell}(v_h) \coloneqq \dfrac{1}{|e|} \int_e v_h 
    \left( \dfrac{{\bf s} - {\bf s}_e}{h_e} \right)^{{\boldsymbol \ell}} {\rm d}s \, , \quad 
    |{\boldsymbol \ell}|\le k-1, 
    \]
    where ${\boldsymbol \ell}\in\N^{d-1}$ is a multi-index with $d-1$ components;
    \item $\pk$: the moments up to the order $k-2$ on each $E \in \Omega_h$ :
    \[
    \mu^{\aaa}_E(v_h) \coloneqq \dfrac{1}{|E|} \int_E v_h \left( \dfrac{\xx - \xx_E}{h_E} \right)^{\aaa} {\rm d}E \, , \quad |\aaa| \leq k-2, 
    \]
    where $\aaa\in\N^d$ is a multi-index with $d$ components.
\end{itemize}
The dimension of $V_h^{k,nc}(\Omega_h)$ is
\[
N = \begin{cases}
k \, |\mathcal E | + \dfrac{(k-1)(k-2)}{2} \, |\Omega_h| \qquad&\text{if $d=2$}\, ,\\[0.3cm]
\dfrac{(k+1)k}{2}\, |\mathcal E |+\dfrac{(k+1)k(k-1)}{6}\, |\Omega_h|\qquad&\text{if $d=3$}\, ,
\end{cases}
\]
where $|\mathcal E |$ and $|\Omega_h|$ are the number of facets and of elements, respectively.
Letting $\vint$ be the DoFs-interpolant of a sufficiently smooth function $v$ into $V_h^{k,nc}(\Omega_h)$, we recall the following interpolation error estimate, see \cite{AML-2016,CMS-2018}.

\begin{lemma}[Approximation with nonconforming virtual element functions]
\label{lm:interpolation}
Under assumption \textbf{(A1)}, for any $v \in H^1(\Omega) \cap H^{s+1}(\Omega_h)$, 
there exists $ \vint \in V_h(\Omega_h)$, such that for all $E \in \Omega_h$, 
\[
\|v - \vint\|_{0,E} + h_E \|\nabla (v - \vint)\|_{0,E} \lesssim h_E^{s+1} |v|_{s+1,E} \, , 
\]
where $0 < s \le k$.
\end{lemma}
\subsection{Virtual Element Forms and the Discrete Problem}\label{s:vemdiscrete}

In this section, discretize the problem and the bilinear forms presented in Section~\ref{sec:continuous_problem}.
As usual in VEM, since we do not have an analytic expression for the functions in $V_h^{k,nc}(\Omega_h)$, the terms entering in \eqref{eq:weak-2} are not computable.
Hence, we have to devise some counterparts of these terms.
First, we note that the forms $a(\cdot,\cdot)$ , 
$b(\cdot,\cdot)$
and 
$c(\cdot,\cdot)$ can be decomposed into local contributions as
\begin{equation}
\label{eq:form-split}
a(u,  v) \eqqcolon \sum_{E \in \Omega_h} a^E(u,  v) \, ,
\quad
b(u,  v) \eqqcolon \sum_{E \in \Omega_h} b^E(u,  v) \, , 
\quad
c(u,  v) \eqqcolon \sum_{E \in \Omega_h} c^E(u,  v) \, .
\end{equation}
The local bilinear form $a^E(\cdot,\cdot)$ is discretized by replacing the gradient with its projection into the space $[\mathbb P_{k-1}(E)]^{d}$. We define the bilinear form $a_h^{E}(\cdot,  \cdot) \colon V_h^{k,nc}(E) \times V_h^{k,nc}(E) \to \R$ as follows:
 \begin{equation} \label{eq:aEh}
 \begin{split}
a_h^E(u_h,  v_h) &:=
\int_E \PZ0P \nabla u_h \cdot \PZ0P \nabla v_h \, {\rm d}E 
+ 
\Stab\bigl((I - \PN) u_h, \, (I - \PN) v_h\bigr) \, ,
\end{split}
\end{equation}
where
$\Stab(\cdot,\cdot)\colon V_h^{k,nc}(E) \times V_h^{k,nc}(E) \to \R$ is 
the standard \texttt{dofi-dofi} stabilization form presented in \cite{volley, BLR:2017},
which is computable and satisfies the fundamental properties
\begin{equation}
\label{eq:sEh}
\begin{aligned}
& \alpha_*|v_h|_{1,E}^2 \leq \Stab(v_h, v_h) 
\qquad &\forall v_h \in V_h^{k,nc}(E) \\
& \Stab(v_h, v_h) \leq \alpha^* |v_h|_{1,E}^2 
\qquad &\forall v_h \in {\rm Ker}(P_E)  \, ,
\end{aligned}
\end{equation}
for two positive constants $\alpha_*$ and $\alpha^*$, independent of $h$. Here, 
$P_E: V_h^{k,nc}(E)\rightarrow \Pk_0(E)$ denotes any projection operator into the space of constant functions.

In the first integral of~\eqref{eq:aEh}, $\nabla\PN$ could be used instead of $\PZ0P \nabla$, since in the case of a constant diffusion coefficient, both choices lead to optimal convergence rates. However, in a general heterogeneous situation, the appropriate 
term to use is $\PZ0P \nabla$, see \cite{BBMR-2016}.

The local convective term $b^E(\cdot,\cdot)$ is discretized by $b_h^{E}(\cdot,  \cdot) \colon V_h^{k,nc}(E) \times V_h^{k,nc}(E) \to \R$ defined as
\begin{equation} \label{eq:bEh}
\begin{split}
b_h^E(u_h,  v_h) &:=
\int_E (\bb \cdot \nabla \P0 u_h) \P0 v_h {\rm d}E \, .
\end{split}
\end{equation}
Finally, the bilinear form $c^E(\cdot,\cdot)$ is discretized by $c_h^{E}(\cdot,  \cdot) \colon V_h^{k,nc}(E) \times V_h^{k,nc}(E) \to \R$ defined as
\begin{equation} \label{eq:cEh}
\begin{split}
c_h^E(u_h,  v_h) &:=
\int_E \P0 u_h \, \P0 v_h {\rm d}E  
+ 
|E| \, \Stab\bigl((I - \P0) u_h, \, (I - \P0) v_h\bigr) \, .
\end{split}
\end{equation}

In order to ensure stability and to impose boundary conditions, we also need to introduce some extra terms.
To get stability, we also introduce the bilinear form $d_h^E(\cdot,\cdot) \colon V_h^{k,nc}(E) \times V_h^{k,nc}(E) \to \R$
\begin{equation} \label{eq:dEh}
\begin{split}
d^E_h(u_h,  v_h) &\coloneqq
-\dfrac{1}{2} 
\int_{\partial E \setminus\Gamma}
\bb\cdot [\![\Pi_k^0 u_h]\!]
\{ \Pi_k^0 v_h \} {\rm d}s \, ,
\end{split}
\end{equation}
The reasons behind this bilinear form will become clear in Lemma~\ref{lm:symm-part} below and are related to the nonconformity of the method. 
Following \cite{burman:2004, burman:2007}, we also introduce a VEM version of the local CIP-stabilization form, in order to ensure stability in the advection-dominated regime.
It is defined as
\begin{equation}
\label{eq:JhE}
J_h^E(u_h,v_h)
\coloneqq 
 \sum_{e \subset \partial E\setminus\Gamma}  \dfrac{\gamma_e}{2} \int_e  \! \, h_e^2 \, [\![\nabla \Pg u_h]\!] \, [\![\nabla \Pg v_h]\!] \, {\rm d}s 
+
\gamma_E \, h_E \, \Stab\bigl((I - \P0) u_h, (I - \P0) v_h\bigr) \, .
\end{equation}
The parameters $\gamma_e$ and $\gamma_E$ are defined as
\begin{equation}\label{eq:gamma-loc}
\gamma_e \coloneqq \kappa_e \| \bb  \|_{L^\infty(e)} \ , \qquad \gamma_E \coloneqq \kappa_E \| \bb  \|_{L^\infty(\partial E)}\, ,
\end{equation}
where $\kappa_e$ and $\kappa_E$ are positive constants to be chosen.
We stress that, whenever $\bb = {\bf 0}$ inside an element $E$, the corresponding parameters $\gamma_e$ and $\gamma_E$ are zero, and hence $J_h^E(u_h, v_h) = 0$.
In order to complete the definition of our method, we need to impose the boundary conditions. We do this using a Nitsche-type technique, which locally consists of adding the form 
\begin{equation}\label{eq:NEh}
\begin{aligned}
\mathcal{N}_h^E(u_h,v_h) \coloneqq
&- \epsilon  \langle \PZ0P \nabla u_h \cdot \nn^E, v_h \rangle_{\partial E\cap\Gamma}
- \epsilon  \langle u_h , \PZ0P \nabla v_h \cdot \nn^E \rangle_{\partial E\cap\Gamma}\\
&+ \dfrac{\epsilon}{\delta h_E}  \sum_{e \subset \partial E\cap\Gamma} \langle \Pi^{0,e}_{k-1} u_h, \Pi^{0,e}_{k-1} v_h \rangle_{e}     
+ \langle \vert \bb \cdot \nn^E \vert \, \P0 u_h, \P0 v_h \rangle_{\partial E\cap \Gamma_{\text{in}}} \, ;
\end{aligned}
\end{equation}
see \cite{Nitsche,Junt-Stenb} for more details.
By adding these terms together, we define the complete local bilinear form $\Acip^E \colon V_h^{k,nc}(E) \times V_h^{k,nc}(E)  \to \R$ as
\begin{equation}
\label{eq:AcipE}
\Acip^E(u_h, v_h) = \epsilon a_h^E(u_h , v_h) + b^E_h(u_h , v_h) + \sigma c_h^E(u_h , v_h) + d^E_h(u_h,v_h) + \mathcal{N}_h^E(u_h,v_h) + J_h^E(u_h , v_h) \, .
\end{equation}
By summing over all the mesh elements, we obtain the global versions of the bilinear forms
\[
\begin{aligned}
&a_h(u_h, v_h) := \sum_{E \in \Omega_h} a_h^E(u_h, v_h)\,,
\qquad 
&b_h(u_h, v_h) := \sum_{E \in \Omega_h} b_h^E(u_h, v_h)\,,
\\
&c_h(u_h, v_h) := \sum_{E \in \Omega_h} c_h^E(u_h, v_h)\,,
\qquad 
&d_h(u_h, v_h) := \sum_{E \in \Omega_h} d_h^E(u_h, v_h)\,,
\\
&J_h(u_h, v_h) := \sum_{E \in \Omega_h} J_h^E(u_h, v_h)\,,
\qquad 
&\mathcal{N}_h(u_h, v_h) := \sum_{E \in \Omega_h} \mathcal{N}_h^E(u_h, v_h)\,,
\end{aligned}
\]
and
\begin{equation}
\label{eq:Acip}
\Acip (u_h, v_h) \coloneqq \sum_{E \in \Omega_h} \AcipE (u_h, v_h) \, .
\end{equation}
The local load term is defined as
\begin{equation}\label{eq:FEh}
\begin{aligned}
\mathcal{F}_h^E(v_h) &\coloneqq \int_E f \, \P0 v_h {\rm d}E 
- \epsilon  \langle g , \PZ0P \nabla v_h \cdot \nn^E \rangle_{\partial E\cap\Gamma} \\
&+ \dfrac{\epsilon}{\delta h_E} 
\sum_{e\subset \partial E\cap \Gamma}
\langle g, \Pi_{k-1}^{0,e}
v_h \rangle_e 
+ \langle \vert \bb \cdot \nn^E \vert \, g, 
\P0 
v_h \rangle_{\partial E\cap\Gamma_{\text{in}}} \, ,
\end{aligned}
\end{equation}
and the global one as
\[
\mathcal{F}_h(v_h) \coloneqq \sum_{E \in \Omega_h} \mathcal{F}_h^E(v_h) \, .
\]
Eventually, the discrete problem reads as follows:
\begin{equation}
\label{eq:cip-vem}
\left \{
\begin{aligned}
& \text{find $u_h \in V^{k,nc}_h(\Omega_h)$ s.t.} 
\\
& \Acip(u_h, \, v_h) = \mathcal{F}_h(v_h) \qquad \text{for all $v_h \in V^{k,nc}_h(\Omega_h)$.}
\end{aligned}
\right.
\end{equation}

\section{Theoretical Analysis}\label{s:theory}

Due to the nonconformity of the virtual element space and due to the projections entering formulation~\eqref{eq:cip-vem}, the solution~$u$ of the continuous problem~\eqref{eq:weak-2} does not solve the discrete problem~\eqref{eq:cip-vem}.
However, if $u \in H^2(\Omega) \cap H_g^1(\Omega)$, it solves the following problem, which is strictly connected to~\eqref{eq:cip-vem}:
\begin{equation}
\label{eq:consistency-1}
\Atcip(u , \, v_h) = \tilde{\mathcal{F}}(v_h) + \tilde{\mathcal{B}}(u, v_h) \qquad \text{for all $v_h \in V^{k,nc}_h(\Omega_h)$} \, .
\end{equation}
The form on the left-hand side of~\eqref{eq:consistency-1} is defined as
\begin{equation}\label{eq:Atcip}
\Atcip(u, v_h) \coloneqq \epsilon \, a(u,v_h) + b(u,v_h) + \sigma \, c(u,v_h) + \tilde{\mathcal{N}}_h(u,v_h) + \tilde J_h(u,v_h) \, ,
\end{equation}
where $\tilde{\mathcal{N}}_h(u,v_h)$ and $\tilde J_h(u,v_h)$ are sums over all mesh elements $E$ of the local contributions
\begin{equation}\label{eq:tildeNEh}
\begin{aligned}
\tilde{\mathcal{N}}_h^E(u,v_h) \coloneqq
&- \epsilon  \langle \nabla u \cdot \nn^E, v_h \rangle_{\partial E\cap\Gamma}
-  \epsilon \langle u, \PZ0P \nabla v_h \cdot \nn^E \rangle_{\partial E\cap\Gamma}
 \\
&+ \dfrac{\epsilon}{\delta h_E} 
\sum_{e \subset \partial E\cap\Gamma} \langle u, \Pi_{k-1}^{0,e} 
v_h \rangle_{e}    
+\langle \vert \bb \cdot \nn^E \vert u, \P0 v_h \rangle_{\partial E\cap\Gamma_\text{in}} \, ,
\end{aligned}
\end{equation}
and 
\begin{equation}\label{eq:tildeJEh}
\tilde J_h^E(u, v_h) 
\coloneqq 
\dfrac{1}{2} \sum_{e \subset \partial E\setminus\Gamma} 
\gamma_e \int_e  \, h_e^2 \, [\![\nabla u]\!]\,  [\![\nabla \Pg v_h]\!] \, {\rm d}s\, ,
\end{equation}
respectively.
On the right-hand side of~\eqref{eq:consistency-1}, the load term $\tilde{\mathcal{F}}(v_h)$ is defined as the sum over all mesh elements $E$ of the local contributions
\begin{equation}\label{eq:tildeFEh}
\begin{aligned}
\tilde{\mathcal{F}}_h^E(v_h) &\coloneqq \int_E f \, v_h {\rm d}E 
-  \epsilon \langle g, \PZ0P \nabla v_h \cdot \nn^E \rangle_{\partial E\cap \Gamma}
\\
&+ \dfrac{\epsilon}{\delta h_E} 
\sum_{e\subset \partial E\cap \Gamma}
\langle g, \Pi_{k-1}^{0,e} 
v_h \rangle_{e}    
+  \langle \vert \bb \cdot \nn^E \vert g, \P0 v_h \rangle_{\partial E\cap\Gamma_\text{in}} \, ,
\end{aligned}
\end{equation}
and form $\tilde{\mathcal{B}}(u,v_h)$, which arises from the nonconformity of the method, is defined as
\[
\tilde{\mathcal{B}}(u,v_h) \coloneqq \sum_{e \in {\mathcal E}_h^0} \epsilon \int_e 
\nabla u\cdot [\![v_h]\!]
{\rm d}s \, .
\]

\subsection{Preliminary results}

In this section, we present some results that are useful in the following analysis. Some of them are already known for conforming virtual element spaces. The first one is an inverse inequality for functions in $V_h^{k,nc}(E)$.
\begin{proposition}[Inverse inequality]
\label{lm:inverse}
Under assumption \textbf{(A1)}, for any $E \in \Omega_h$ we have   
\[
| v_h |_{1,E} \lesssim h_E^{-1} \| v_h \|_{0,E} 
\quad \forall v_h \in V_h^{k,nc}(E) .
\]
\end{proposition}
\begin{proof}
An integration by parts gives
\begin{equation}\label{eq:inverse-zero}
| v_h |^2_{1,E} 
= 
\int_E \nabla v_h \cdot \nabla v_h \, {\rm d}E 
= - \int_E v_h \, \Delta v_h {\rm d}E
+ \int_{\partial E} v_h \, \nabla v_h \cdot \nn^E  \, {\rm d}s . 
\end{equation}
Since $\Delta v_h$ is polynomial we have (see for instance \cite{BLR:2017}\footnote{Several of the results we quote have been proven in the references for $d=2$. Nevertheless, these results can be extended to $d=3$ in a straightforward manner.})
\begin{equation}\label{eq:inv-lapl}
\| \, \Delta v_h \|_{0,E}\le C_\Delta h_E^{-1} |  v_h |_{1,E},   
\end{equation}
for a constant~$C_\Delta>0$ independent of $h_E$.
Therefore, for the first term on the right-hand side of \eqref{eq:inverse-zero}, we get
\begin{equation} \label{eq:inverse-first}
- \int_E v_h \, \Delta v_h \,  {\rm d}E
\leq
\| v_h \|_{0,E} \| \, \Delta v_h \|_{0,E}
\le C_\Delta
h^{-1}_E \, \| v_h \|_{0,E} |  v_h |_{1,E} \, .
\end{equation}
For the second term, using the Cauchy-Schwarz inequality and the multiplicative trace inequality, we obtain 
\begin{equation}\label{eq:inv-nd-1}
\begin{aligned}
\int_{\partial E} v_h \, \nabla v_h \cdot \nn^E \, {\rm d}s
& \leq
\left \| v_h \right \|_{0,\partial E} \, 
\left \| \nabla v_h \cdot \nn^E \right \|_{0,\partial E} \\
& \leq C_{t_1}
\| v_h \|_{0,E}^{\frac{1}{2}} \,
\left( h^{-1}_E\| v_h\|_{0,E} + | v_h|_{1,E}\right)^{\frac{1}{2}} \,
\left \| \nabla v_h \cdot \nn^E \right \|_{0,\partial E} \, ,
\end{aligned}
\end{equation}
for a positive constant $C_{t_1} > 0$ independent of $h_E$. 
The multiplicative trace inequality and~\eqref{eq:inv-lapl} give 
\begin{equation}\label{eq:inverse-two}
\begin{split}
\left \| \nabla v_h \cdot \nn^E \right \|_{0,\partial E} &\le
C_{t_2} | v_h |_{1,E}^{\frac{1}{2}}
\left(h_E^{-1}| v_h |_{1,E}+\| \Delta v_h \|_{0,E}\right)^{\frac{1}{2}}\\
&\le C_{t_2} | v_h |_{1,E}^{\frac{1}{2}}
\left(h_E^{-1}| v_h |_{1,E}+C_\Delta h_E^{-1}| v_h |_{1,E}\right)^{\frac{1}{2}}\\
&\le C_{t_2}\sqrt{1+C_\Delta}h_E^{-\frac{1}{2}}| v_h |_{1,E},
\end{split}
\end{equation}
for $C_{t_2} > 0$ independent of $h_E$.
Hence, with~$C_t:=\max\{C_{t_1}, C_{t_2}\}$, we obtain 
\begin{equation}\label{eq:inv-nd3}
\begin{aligned}
\int_{\partial E} v_h \, \nabla v_h \cdot \nn^E \, {\rm d}s
& \le
C_t^2\sqrt{1+C_\Delta}\, h^{-\frac{1}{2}}_E\| v_h \|_{0,E}^{\frac{1}{2}} \,
\left( h^{-1}_E\| v_h\|_{0,E} + | v_h|_{1,E}\right)^{\frac{1}{2}} |v_h|_{1,E}.
\end{aligned}
\end{equation}
From~\eqref{eq:inverse-zero},  \eqref{eq:inverse-first}, and~\eqref{eq:inv-nd3}, using the Young inequality~$ab\le \frac{1}{2\eta}a^2+\frac{\eta}{2}b^2$ for any $\eta>0$,
we get 
\[
\begin{split}
| v_h |_{1,E} &\le 
C_\Delta h_E^{-1}\| v_h\|_{0,E}
+ C_t^2\sqrt{1+C_\Delta}\, h^{-\frac{1}{2}}_E\| v_h \|_{0,E}^{\frac{1}{2}}\left( h^{-1}_E\| v_h\|_{0,E} + | v_h|_{1,E}\right)^{\frac{1}{2}} \\
&\le C_\Delta h_E^{-1}\| v_h\|_{0,E}
+ C_t^2\sqrt{1+C_\Delta}
\left(\frac{1}{2\eta}h^{-1}_E\| v_h\|_{0,E} +\frac{\eta}{2}h^{-1}_E\| v_h\|_{0,E}+\frac{\eta}{2}| v_h|_{1,E}\right).
\end{split}
\]
Choosing $\eta=1/(C_t^2\sqrt{1+C_\Delta})$, we conclude
\[
| v_h |_{1,E}\le \left(2C_\Delta+C_t^4(1+C_\Delta)+1\right)
h^{-1}_E\| v_h\|_{0,E},
\]
and the proof is complete.
\end{proof}
The second result is an inverse trace inequality for functions in $V_h^{k,nc}(E)$ with internal DoFs equal to zero. 
\begin{lemma} [Inverse trace inequality]
\label{lm:invtrace}
Under assumption \textbf{(A1)}, for any $E \in \Omega_h$ 
we have
\[
\| v_h \|_{0,E} 
\lesssim 
\left( h_E \sum_{e \subset \partial E} \| \Pi^{0,e}_{k-1} v_h \|_{0, e}^2 \right)^{\frac{1}{2}}\quad \text{$\forall v_h \in V_h^{k,nc}(E)$ such that $\Pi^{0,E}_{k-2} v_h \equiv 0$.}
\]
\end{lemma}
\begin{proof}
Thanks to the orthogonality of the $\P0$ projection in $L^2(E)$, we have
\begin{equation} \label{eq:invtr-0}
\| v_h \|_{0,E}^2 = \| (I - \P0) v_h \|_{0,E}^2 + \| \P0 v_h \|_{0,E}^2 \, . 
\end{equation}
We now remark that the techniques leading to \cite[Lemma~2.18]{brenner-sung:2018} also apply to the nonconforming space $V_h^{k,nc}(E)$. Therefore, using $\Pi^{0,E}_{k-2} v_h = 0$, we have
\begin{equation} \label{eq:invtr-1}
\|\P0 v_h \|_{0,E}^2 \lesssim h_E \sum_{e \subset \partial E} \| \Pi^{0,e}_{k-1} v_h \|_{0,e}^2 \, .
\end{equation}
In addition, Lemma \ref{lm:bramble} gives
\begin{equation} \label{eq:invtr-2}
\| (I - \P0) v_h \|_{0,E}^2 \lesssim h_E^2 \, | v_h |_{1,E}^2 \, .
\end{equation}
Using integration by parts, recalling that $\Delta v_h\in \Pk_k(E)$ and $(\nabla v_h \cdot \nn^E){}_{|_e}\in \Pk_{k-1}(e)$, 
exploiting estimates \eqref{eq:inv-lapl}, \eqref{eq:invtr-1} and \eqref{eq:inverse-two} we obtain
\begin{equation} \label{eq:invtr-3}
\begin{aligned}
| v_h |^2_{1,E}
&= - \int_E v_h \, \Delta v_h {\rm d}E
+ \int_{\partial E} v_h \, \nabla v_h \cdot \nn^E \, {\rm d}s \\
&=
- \int_E \P0 v_h \, \Delta v_h {\rm d}E
+ \sum_{e \subset \partial E} \int_{e} \Pi^{0,e}_{k-1} v_h \, \nabla v_h \cdot \nn^E \, {\rm d}s \\
&\leq 
\| \P0 v_h \|_{0,E} \| \Delta v_h \|_{0,E}
+
\sum_{e \subset \partial E}
\| \Pi^{0,e}_{k-1} v_h \|_{0,e} \left \| \nabla v_h \cdot \nn^E \right \|_{0,e} \\
&\lesssim
\left( h^{-1}_E \sum_{e \subset \partial E}
\| \Pi^{0,e}_{k-1} v_h \|_{0,e}^2 \right)^{\frac{1}{2}} |v_h|_{1,E} .
\end{aligned}
\end{equation}
It follows that 
\begin{equation} \label{eq:invtr-4}
| v_h |^2_{1,E} \lesssim h^{-1}_E \sum_{e \subset \partial E}
\| \Pi^{0,e}_{k-1} v_h \|_{0,e}^2 .
\end{equation}
Hence, from \eqref{eq:invtr-2} and~\eqref{eq:invtr-4}, we get
\begin{equation} \label{eq:invtr-5}
\| (I - \P0) v_h \|_{0,E}^2 \lesssim h_E \sum_{e \subset \partial E}
\| \Pi^{0,e}_{k-1} v_h \|_{0,e}^2 \, .
\end{equation}
Collecting \eqref{eq:invtr-0}, \eqref{eq:invtr-1}, and \eqref{eq:invtr-5} concludes the proof.
\end{proof}
We now construct an Owsald-type interpolation operator $\pi$ that maps piecewise (sufficiently) smooth functions into the nonconforming space $V_h^{k,nc}(\Omega_h)$. 
Recall the definitions~\eqref{eq:dofsInt} and~\eqref{eq:dofsEdges} of the local DoFs.
To define the global DoFs, we thus average the local DoFs at interelement boundaries. 
More precisely, the interpolant is constructed as 
\begin{equation}
\label{eq:oswald-inter-1}
\pi v= \sum_{e \in \mathcal{E}_h} \sum_{|{\boldsymbol \ell|\le k-1}}\mu_e^{{\boldsymbol \ell}} (v) \phi_e^{{\boldsymbol \ell}} + \sum_{E \in \Omega_h} \sum_{|\aaa| \leq k-2} \mu^{\aaa}_E(v) \varphi_E^{\aaa} \, ,
\end{equation}
where $\{ \varphi_e^{{\boldsymbol \ell}} \}$ is the set of basis functions associated to the DoFs at the mesh skeleton $\ek$, and $\{ \varphi_E^{\aaa} \}$ is the set of basis functions associated to the interior DoFs. For an interior facet $e\in \mathcal{E}_h^0$, $e\subset\partial E\cap \partial E'$, the coefficient $\mu_e^{{\boldsymbol \ell}} (v)$ is defined as 
\begin{equation}
\label{eq:oswald-inter-2}
\mu_e^{{\boldsymbol \ell}} (v)  
\coloneqq  
\dfrac{1}{2} \left( \mu_{E,e}^{{\boldsymbol \ell}} (v) +  \mu_{E',e}^{{\boldsymbol \ell}} (v)\right),
\end{equation}
while for a boundary edge $e\in \mathcal{E}_h\setminus \mathcal{E}_h^0$, it is simply defined as
\begin{equation}
\label{eq:oswald-inter-2bis}
\mu_e^{{\boldsymbol \ell}} (v)  
\coloneqq  
\mu_{E,e}^{{\boldsymbol \ell}} (v) .
\end{equation}
Therefore, considering \eqref{eq:jmp-avrg} and \eqref{eq:dofsEdges}, on any edge $e$ we get
\begin{equation}\label{eq:avrg-oswald}
\Pi^{0,e}_{k-1}(\pi v) =  \{ \Pi^{0,e}_{k-1} v  \}  
= \Pi^{0,e}_{k-1} (\{ v  \}) .
\end{equation}
The main result of this section is the following theorem.
\begin{theorem}\label{prp:clmest}
Let $p \in \mathbb{P}_{k}(\Omega_h)$ be a (discontinuous) piecewise polynomial, and let $\pi: \mathbb P_{k}(\Omega_h) \to V_h^{k,nc}(\Omega_h)$ denotes the Oswald's interpolant. We have that
\[
\| (I - \pi) \, p \|_{0,E}^2
\lesssim
h_E \sum_{e \subset \partial E\setminus\Gamma} \| [\![p]\!] \|_{0,e}^2 \qquad \forall E \in \Omega_h \, , \quad  \forall p \in \Pk_{k}(\Omega_h) \, .
\]
\end{theorem}
\begin{proof}
We introduce the difference
\[
d \coloneqq (I - \pi ) p \, .
\]
We restrict our attention to an element $E\in\Omega_h$, and consider $d^E \coloneqq d|_E$.
From \eqref{eq:avrg-oswald} we get
\begin{equation}\label{eq:jmp-eqn}
\Pi^{0,e}_{k-1} d^E = \left\{
\begin{aligned}
&\frac{1}{2} \Pi^{0,e}_{k-1} ([\![p]\!]  \cdot \nn^E )&\mbox{ if $e\not\subset\Gamma$},\\  
& 0 &\mbox{ if $e \subset\Gamma$}. 
\end{aligned}
\right.
\end{equation}
We now observe that the interior DoFs of $d^E$ are equal to zero, so that $\Pi^{0,E}_{k-2}d^E = 0$, see \eqref{eq:dofsInt}. Hence, from Lemma~\ref{lm:invtrace} and~\eqref{eq:jmp-eqn}, we obtain
\begin{equation}\label{1}
\begin{aligned}
\| d^E \|_{0,E} 
\lesssim 
&\left(  h_E \sum_{e \subset \partial E\setminus\Gamma}\| \Pi^{0,e}_{k-1} d^E \|_{0,e}^2 \right)^{\frac{1}{2}}
\lesssim 
\left(  h_E  \sum_{e \subset \partial E\setminus\Gamma}  \| \Pi^{0,e}_{k-1} ([\![p]\!]  \cdot \nn^e )\|_{0,e}^2   \right)^{\frac{1}{2}} \\
\lesssim 
&\left( h_E \sum_{e \subset \partial E\setminus\Gamma}  \| [\![p]\!]  \cdot \nn^e \|_{0,e}^2   \right)^{\frac{1}{2}} =
\left( h_E  \sum_{e \subset \partial E\setminus\Gamma}  \| [\![p]\!]   \|_{0,e}^2   \right)^{\frac{1}{2}}  \, .
\end{aligned}
\end{equation}
\end{proof}

\begin{lemma}
\label{lm:clmcont}
Under assumption \textbf{(A1)}, for every $E \in \Omega_h$, we have
\[
\| \pi p \|_{0,E} \lesssim \| p \|_{0,\mathcal{D}(E)} \quad \text{for all $p \in \Pk_k(\Omega_h)$}\, ,
\]
where $\mathcal{D}(E) \coloneqq \bigcup \{ K \in \Omega_h \quad \text{s.t.}\quad | \partial E \cap \partial K |  > 0 \}.$
\end{lemma}

\begin{proof}
As in Lemma 3.3 of \cite{BLT:2024}, the proof easily follows from the triangle inequality 
$\| \pi p \|_{0,E}\le \| \pi p - p \|_{0,E} + \|  p \|_{0,E}$, together with Theorem \ref{prp:clmest} and a polynomial trace inequality.   \end{proof}

\begin{remark}
We remark that Theorem \ref{prp:clmest} and Lemma \ref{lm:clmcont} are actually valid also for $v\in L^2(\Omega)$ such that $v_{|E}\in V_h^{k,nc}(E)$ for every $E\in\Omega_h$. In particular, in this case, the proof of Lemma \ref{lm:clmcont} requires the application of the Agmon inequality and Proposition \ref{lm:inverse}, instead of the polynomial trace inequality.    
\end{remark}

\subsection{Inf-Sup condition}
The goal of this section is to prove the following inf-sup condition for the discrete problem~\eqref{eq:cip-vem}, namely
\begin{equation} \label{eq:infsup}
\| v_h \|_{\cip} 
\lesssim 
\sup_{z_h \in V_h^{k,nc}(\Omega_h)} \dfrac{\Acip (v_h, z_h)}{\| z_h \|_{\cip}}
\qquad \forall v_h \in V_h^{k,nc}(\Omega_h)\, ,
\end{equation}
where the norm 
$\| \cdot \|_{\cip} $ in~$V^{k,nc}_h(\Omega_h)$ is defined by
\begin{equation} \label{eq:glb_norm_def}
\|v_h\|^2_{\cip} \coloneqq \sum_{E \in \Omega_h} \|v_h\|^2_{\cip, E} \,,
\end{equation}
with 
\begin{equation} \label{eq:loc_nom_def}
\begin{aligned}
\|v_h\|^2_{\cip  , E} 
:= &
\epsilon \, \| \nabla v_h \|^2_{0,E} + 
h \, \| \bb \cdot \nabla \P0 v_h \|^2_{0,E} + 
\sigma \, \| v_h \|^2_{0,E} \\
&+ 
\dfrac{\epsilon}{\delta h} \sum_{e\subset\partial E\cap \Gamma} 
\|\Pi_{k-1}^{0,e}v_h\|_{0,e}^2 +
\| | \bb \cdot \nn^E |^{\frac{1}{2}} \P0 v_h \|^2_{0,\partial E\cap\Gamma_\text{in}}
+
J_h^E(v_h,v_h) \, .
\end{aligned}
\end{equation}

We divide the proof of~\eqref{eq:infsup} in two parts. In the first part (Lemma~\ref{lm:symm-part}), we estimate the diffusion, reaction, and inflow boundary terms in $\| v_h \|_{\cip}^2$ with $\Acip(v_h, v_h)$.
In the second part (Lemma~\ref{lm:adv-term}),
we estimate the convective term in $\| v_h \|_{\cip}^2$. In order to do so, we would like to take $\widetilde w_h$ locally defined by
\[
\widetilde w_h |_{E} \coloneqq h\bb \cdot \nabla \P0 v_h 
\]
as the second argument in $\Acip(\cdot, \cdot)$.
This is not possible, since $\widetilde w_h \not\in V^{k,nc}_h(\Omega_h)$, and we take its Oswald interpolant $\pi\widetilde w_h$ instead. Then, the difference between $\widetilde w_h$ and $\pi\widetilde w_h$ is controlled thanks to the jump term. We combine these two results and conclude the proof of the inf-sup condition~\eqref{eq:infsup} in Theorem~\ref{thm:inf-sup}.

\begin{lemma}\label{lm:symm-part}
Under assumptions \textbf{(A1)}, given $v_h\in V^{k,nc}_h(\Omega_h)$, we have 
\begin{equation*}\label{eq:infsupsymmetry_gbl}
\begin{split}
\Acip(v_h, v_h) \gtrsim & \epsilon 
\| \nablah v_h \|_{0}^2
+
\sigma \| v_h \|_{0}^2
+
\sum_{E\in\Omega_h}\dfrac{\epsilon}{\delta h} \sum_{e\subset\partial E\cap \Gamma} 
\|\Pi_{k-1}^{0,e}v_h\|_{0,e}^2\\
&+\sum_{E\in\Omega_h} \||\bb\cdot\nn^E|^{\frac{1}{2}}\P0 v_h\|_{0,\partial E\cap \Gamma_{{\rm in}}}^2
+ 
J_h(v_h, v_h) 
 \, . 
 \end{split}
\end{equation*}
\end{lemma}

\begin{proof}
We take the same $v_h$ in both arguments of $\Acip(\cdot,\cdot)$ and obtain
\[
\Acip(v_h, v_h) = \epsilon \, a_h(v_h,v_h) + b_h(v_h,v_h) + \sigma \, c_h(v_h,v_h) + d_h(v_h,v_h) + \mathcal{N}_h(v_h,v_h) + J_h(v_h,v_h) \, .
\]

The diffusion term is easily estimated using the orthogonality of the projectors and property~\eqref{eq:sEh} of
the stabilization form: 
\[
\epsilon a_h(v_h, v_h)\geq \widetilde\alpha_* \epsilon\, \| \nablah v_h \|^2_{0}, 
\]
where $\widetilde\alpha_* = \min\{1, \alpha_* \}$. For the reaction term, using \eqref{eq:sEh} and the Poincar\'e inequality for $(I - \P0) v_h$, we have
\[
\sigma c_h(v_h, v_h) \geq C_0 \widetilde\alpha_* \sigma \| v_h \|^2_{0} .
\]
for a positive constant $C_0$ independent of $\sigma$ and $h$.
Therefore, we get
\begin{equation}\label{eq:est_ac}
\epsilon a_h(v_h, v_h) + \sigma c_h(v_h, v_h) \geq \widetilde\alpha_* \epsilon\, \| \nablah v_h \|^2_{0} + C_0 \widetilde\alpha_*\sigma \| v_h \|^2_{0} \, .
\end{equation}

For the convective term, by integrating by parts, we obtain
\[
\begin{aligned}
b_h(v_h, v_h) 
 =&
\dfrac{1}{2}
\int_\Omega (\bb \cdot \nablah \Pg v_h) \Pg v_h {\rm d}E
-
\dfrac{1}{2}\int_\Omega \Pg v_h (\bb \cdot \nablah \Pg v_h){\rm d}E \\
&+
\dfrac{1}{2} \sum_{E \in \Omega_h}\int_{\partial E} (\bb \cdot \nn^E) \P0 v_h \P0 v_h {\rm d} s \, \\
=&\dfrac{1}{2} \sum_{E\in \Omega_h} 
\int_{\partial E\cap\Gamma} (\bb \cdot \nn^E) \P0 v_h \P0 v_h {\rm d} s
-d_h(v_h,v_h),
\end{aligned}
\]
where the last step follows from simple algebraic manipulations and from the definition of $d_h(\cdot,\cdot)$.
Combining the terms on $\Gamma_{\rm in}$ in the expression above with those in $\mathcal{N}_h(\cdot,\cdot)$ gives
\[
\begin{aligned}
b_h(v_h,v_h)&+d_h(v_h,v_h)+ \mathcal{N}_h(v_h,v_h) = \dfrac{1}{2} \sum_{E\in\Omega_h} 
\||\bb\cdot\nn^E|^{\frac{1}{2}}\P0 v_h\|_{0,\partial E\cap \Gamma}^2\\
&+  \sum_{E\in\Omega_h}\dfrac{\epsilon}{\delta h_E} \sum_{e\subset\partial E\cap \Gamma} 
\|\Pi_{k-1}^{0,e}v_h\|_{0,e}^2
-2\epsilon \sum_{E\in\Omega_h} \langle \PZ0P\nabla v_h\cdot\nn^E,v_h\rangle_{\partial E\cap\Gamma}\, .
\end{aligned}
\]
For the last term on the right-hand side, we use the estimate 
\[
\begin{aligned}
2\epsilon & \langle \PZ0P\nabla v_h\cdot\nn^E,v_h\rangle_{\partial E\cap\Gamma}
=2\epsilon \sum_{e\subset \partial E\cap\Gamma}\langle \PZ0P\nabla v_h\cdot\nn^E,\Pi_{k-1}^{0,e}v_h\rangle_{e}\\
&\le 2\epsilon\delta h_E\|\PZ0P\nabla v_h\cdot\nn^E\|_{\partial E\cap\Gamma}^2+ \dfrac{\epsilon}{2\delta h_E} \sum_{e\subset \partial E\cap\Gamma}\|\Pi_{k-1}^{0,e}v_h\|_{0,e}^2\\
&\le  2\epsilon\delta C_{\rm tr}\|\PZ0P\nabla v_h\|_{0,E}^2+ \dfrac{\epsilon}{2\delta h_E} \sum_{e\subset \partial E\cap\Gamma} \|\Pi_{k-1}^{0,e}v_h\|_{ 0,e}^2\\
&\le 2\epsilon\delta C_{\rm tr}\|\nabla v_h\|_{0,E}^2+ \dfrac{\epsilon}{2\delta h_E} \sum_{e\subset \partial E\cap\Gamma}\|\Pi_{k-1}^{0,e}v_h\|_{0,e}^2\, ,
\end{aligned}
\]
 where $C_{\rm tr}$ is the inverse trace inequality constant for polynomials. 
Therefore, with the choice, e.g., $\delta=\widetilde\alpha_*/(4 C_{\rm tr})$,  we obtain
\[
\begin{aligned}
b_h(v_h,v_h)+d_h(v_h,v_h)+ \mathcal{N}_h(v_h,v_h)\ge & \dfrac{1}{2} \sum_{E\in\Omega_h} 
\||\bb\cdot\nn^E|^{\frac{1}{2}}\P0 v_h\|_{0,\partial E\cap \Gamma}^2\\
&
+\frac12 \sum_{E\in\Omega_h}\dfrac{\epsilon}{\delta h_E} \sum_{e\subset\partial E\cap \Gamma} 
\|\Pi_{k-1}^{0,e}v_h\|_{0,e}^2
-\dfrac{\widetilde\alpha_*}{2}\epsilon\|\nablah v_h\|_{0}^2
\, .
\end{aligned}
\]
This, together with~\eqref{eq:est_ac}, 
and noting that
$$
\sum_{E\in\Omega_h} \||\bb\cdot\nn^E|^{\frac{1}{2}}\P0 v_h\|_{0,\partial E\cap \Gamma}^2\geq
\sum_{E\in\Omega_h} \||\bb\cdot\nn^E|^{\frac{1}{2}}\P0 v_h\|_{0,\partial E\cap \Gamma_{\rm in}}^2,
$$
gives the result.
\end{proof}

\begin{lemma}\label{lm:adv-term}
Given $v_h\in V_h^{k,nc}(\Omega_h)$, let us define the function
\begin{equation} \label{eq:wh}
w_h \coloneqq \pwh \, ,
\end{equation}
where $\bb_h$ is the $L^2$-projection of $\bb$ in the space of piecewise linear functions $[\Pk_1(\Omega_h)]^d$. Then, under assumptions \textbf{(A1)}, if the mesh size satisfies 
$h>\varepsilon$, 
we have 
\begin{equation}\label{eq:adv-term-control}
\Acip(v_h, w_h) 
\geq
C_1\, h \| \bb \cdot \nablah \Pg v_h \|^2_{0} 
-C_2\, \, \Acip(v_h, v_h) \, ,
\end{equation}
for $C_1,C_2>0$ independent of $\epsilon$ and $h$.
\end{lemma}
\begin{proof}
From Lemma \ref{lm:clmcont}, $w_h$ defined in~\eqref{eq:wh} satisfies the following estimate, which will be used throughout the rest of the proof:
\begin{equation}\label{eq:pi_estimate}
\|w_h\|_{0,E}
\lesssim
h\| \bb_h \cdot \nablah \Pg v_h \|_{0,\mathcal D(E)} \, ,
\end{equation}
We proceed element by element. 
By definition of the local bilinear form $\AcipE(\cdot, \cdot)$, we have that
\begin{equation}\label{eq:new_loc_A}
\begin{aligned}
\AcipE (v_h, w_h) 
& =
\epsilon \, a_h^E (v_h, w_h) 
+ J_h^E (v_h, w_h) 
  +
\sigma \, c_h^E(v_h, w_h)\\ 
& \quad
+
\mathcal{N}^E_h(v_h, w_h)  + 
b_h^E (v_h, w_h)
+
d_h^E (v_h, w_h)\\
& =: T_1 + T_2 + T_3 + T_4 + T_5 + T_6\, .
\end{aligned}
\end{equation}
We estimate each of these six terms separately.\\
{\it Estimate of} $\mathbf{(T_1).}$ 
Using the Cauchy-Schwarz inequality, the properties of $a_h(\cdot, \cdot)$, the inverse inequality for virtual functions, estimate~\eqref{eq:pi_estimate}, and recalling that $\epsilon < h$, we get
\begin{equation} \label{eq:ah_infsup}
\begin{aligned}
T_1=\epsilon \, a_h^E (v_h, w_h) 
& \geq
- \epsilon \, a_h^E(v_h, v_h)^{\frac{1}{2}} \, a_h^E(w_h, w_h)^{\frac{1}{2}} \\
&\gtrsim 
-\epsilon^{\frac{1}{2}} \| \nabla v_h \|_{0,E} \, \epsilon^{\frac{1}{2}} \| \nabla w_h \|_{0,E} \\
 & \gtrsim 
-\epsilon^{\frac{1}{2}} \| \nabla v_h \|_{0,E} \, \epsilon^{\frac{1}{2}} h^{-1} \| w_h \|_{0,E} \\
&\gtrsim 
-\epsilon^{\frac{1}{2}} \| \nabla v_h \|_{0,E} \, h^{\frac{1}{2}} \| \bb_h \cdot \nablah \Pg v_h \|_{0,\mathcal{D}(E)}  \, .
\end{aligned}
\end{equation}
{\it Estimate of} $\mathbf{(T_2).}$
We first split $J_h^E(\cdot,\cdot)$ using the Cauchy-Schwarz inequality
\begin{equation*} \label{eq:J_partial}
T_2=J_h^E(v_h, w_h) 
\geq
-J_h^E(v_h, v_h)^{\frac{1}{2}} \, J_h^E(w_h, w_h)^{\frac{1}{2}} \, .
\end{equation*} 
Using inverse and inverse trace inequalities for polynomials, the inverse inequality for virtual functions, the properties of the $L^2$ projection, and \eqref{eq:pi_estimate}, we get
\begin{equation*} 
\begin{aligned}
J_h^E(w_h, w_h)  
&=
\dfrac{1}{2} \sum_{e \subset \partial E\setminus \Gamma} \gamma_e \int_e \, h_e^2 \, |[\![\nablah \Pg w_h]\!]|^2 \, {\rm d}s
+
\gamma_E \, h_E \, \Stab_J \bigl( ( I - \PN ) w_h, ( I - \PN ) w_h \bigr) \\
& \lesssim
 h \, \| \nablah \Pg w_h \|_{0,\mathcal{D}(E)}^2
+  h \, \| \nabla w_h \|_{0,E}^2 \\
& \lesssim
  h^{-1} \, \|  \Pg w_h \|_{0,\mathcal{D}(E)}^2 +  h^{-1} \| w_h \|_{0,E}^2 \\
& \lesssim
h^{-1} \| w_h \|_{0,\mathcal{D}(E)}^2 
\lesssim
h \| \bb_h \cdot \nablah \Pg v_h \|_{0,\mathcal{D}(\mathcal{D}(E))}^2  \, ,
\end{aligned}
\end{equation*}
from which we conclude
\begin{equation}\label{eq:Jhinfsup}
T_2\gtrsim -J_h^E(v_h, v_h)^{\frac{1}{2}}
h^{\frac{1}{2}}\| \bb_h \cdot \nablah \Pg v_h \|_{0,\mathcal{D}(\mathcal{D}(E))}\, .
\end{equation}

{\it Estimate of} $\mathbf{(T_3).}$
From the properties of $c_h^E(\cdot, \cdot)$ and~\eqref{eq:pi_estimate}, we get 
\begin{equation} \label{eq:chinfsup}
\begin{aligned}
T_3=\sigma c_h^E(v_h, w_h) 
&\gtrsim  
- \sigma \| v_h \|_{0,E} \, \| w_h \|_{0,E} \\
&\gtrsim
- \| v_h \|_{0,E} \,  h^{\frac{1}{2}} \| \bb_h \cdot \nablah \Pg v_h \|_{0,\mathcal{D}(E)} \, . 
\end{aligned}
\end{equation}
where we used $h^{\frac{1}{2}} \lesssim 1$ to simplify later developments.

\noindent
{\it Estimate of} $\mathbf{(T_4).}$
For the Nitsche term, we have to control four different terms:
\[
\begin{aligned}
T_4=\mathcal{N}^E_h(v_h,w_h) & =
-  \epsilon \langle \PZ0P\nabla v_h \cdot \nn^E, w_h \rangle_{\partial E\cap\Gamma}
-  \epsilon \langle v_h, \PZ0P\nabla  w_h \cdot \nn^E \rangle_{\partial E\cap\Gamma}
  \\
&\quad + \dfrac{\epsilon}{\delta h_E}\sum_{e\subset\partial E\cap\Gamma} \langle \Pi_{k-1}^{0,e} v_h, \Pi_{k-1}^{0,e} w_h \rangle_{e}    
+ 
\langle \vert \bb \cdot \nn^E \vert \P0 v_h, \P0 w_h \rangle_{\partial E\cap\Gamma_{\rm in}} \\
& \eqqcolon
\eta_{\mathcal{N}_1} + \eta_{\mathcal{N}_2} + \eta_{\mathcal{N}_3} + \eta_{\mathcal{N}_4} \, .
\end{aligned}
\]
We remark that $\mathcal{N}^E_h(\cdot,\cdot)$ is different from zero only if $E$ has at least one facet on $\Gamma$.\\
For $\eta_{\mathcal{N}_1}$, we use Chauchy-Schwarz inequality, the inverse trace inequality for polynomials and for virtual functions, esitmate~\eqref{eq:pi_estimate}, and the assumption $\epsilon < h$, and derive
\begin{equation}\label{eq:eta1}
\begin{aligned}
\eta_{\mathcal{N}_1}
& \gtrsim
- \epsilon \| \PZ0P\nabla v_h \|_{0,\partial E\cap\Gamma} \| w_h \|_{0,\partial E\cap\Gamma} \\
& \gtrsim
- \epsilon \, h^{-1} \| \PZ0P\nabla v_h \|_{0,E} \| w_h \|_{0,E} \\
& \gtrsim
- \epsilon^{\frac{1}{2}} \| \nabla v_h \|_{0,E} \, h^{\frac{1}{2}} \, \| \bb_h \cdot \nablah \Pg v_h \|_{0,\mathcal{D}(E)}\, .
\end{aligned}
\end{equation}
For $\eta_{\mathcal{N}_2}$, we use the orthogonality of $\Pi^{0,e}_{k-1}$, the Cauchy-Schwarz inequality, inverse trace and inverse inequalities, and $\epsilon < h$, to obtain
\begin{equation}\label{eq:eta2}
\begin{aligned}
\eta_{\mathcal{N}_2}
& =
-  \epsilon \sum_{e\subset\partial E\cap\Gamma}
\langle \Pi_{k-1}^{0,e} v_h, \PZ0P\nabla w_h \cdot \nn^E \rangle_{e} \\
& \gtrsim
- \epsilon \sum_{e\subset\partial E\cap\Gamma}\| \Pi_{k-1}^{0,e} v_h \|_{0,e} \| \PZ0P\nabla w_h \|_{0,\partial E\cap\Gamma} \\
& \gtrsim
- \epsilon \, h^{-\frac{1}{2}} \sum_{e\subset\partial E\cap\Gamma}\| \Pi_{k-1}^{0,e} v_h \|_{0,e} \| \PZ0P\nabla w_h \|_{0,E} \\
& \gtrsim
- \epsilon \, h^{-\frac{3}{2}} \sum_{e\subset\partial E\cap\Gamma}\| \Pi_{k-1}^{0,e} v_h \|_{0,e} \| w_h \|_{0,E} \\
&\gtrsim
- \left(\dfrac{\epsilon}{\delta h}\right)^{\frac{1}{2}} \sum_{e\subset\partial E\cap\Gamma}\| \Pi_{k-1}^{0,e} v_h \|_{0,e} \, h^{\frac{1}{2}} \| \bb_h \cdot \nablah \Pg v_h \|_{0,\mathcal{D}(E)} \, .
\end{aligned}
\end{equation}
For $\eta_{\mathcal{N}_3}$, thanks to orthogonality, we remove the $\Pi^{0,e}_{k-1}$ projection on the second term and we proceed similarly to the previous cases:
\begin{equation}\label{eq:eta3}
\begin{aligned}
\eta_{\mathcal{N}_3}
& = 
\dfrac{\epsilon}{\delta h}\sum_{e\subset\partial E\cap\Gamma} \langle \Pi_{k-1}^{0,e} v_h, \ w_h \rangle_{e} 
\gtrsim 
- \dfrac{\epsilon}{\delta h} \sum_{e\subset\partial E\cap\Gamma}\| \Pi_{k-1}^{0,e} v_h \|_{0,e} \| w_h \|_{0,\partial E\cap\Gamma} \\
& \gtrsim
- h^{-1}
\left(\dfrac{\epsilon}{\delta h}\right)^{\frac{1}{2}}
\sum_{e\subset\partial E\cap\Gamma}\| \Pi_{k-1}^{0,e} v_h \|_{0,e} \| w_h \|_{0,E} \\
& \gtrsim
- \left(\dfrac{\epsilon}{\delta h}\right)^{\frac{1}{2}} \sum_{e\subset\partial E\cap\Gamma}\| \Pi_{k-1}^{0,e} v_h \|_{0,e} \, h^{\frac{1}{2}} \| \bb_h \cdot \nablah \Pg v_h \|_{0,\mathcal{D}(E)} \, ,
\end{aligned} 
\end{equation}
where we have used again $h\lesssim 1$.\\
For $\eta_{\mathcal{N}_4}$, using $| \bb \cdot \nn^E |\le 1$, we have that 
\begin{equation}\label{eq:eta4}
\begin{aligned}
\eta_{\mathcal{N}_4}
& \gtrsim
- \| | \bb \cdot \nn^E |^{\frac{1}{2}} \P0 v_h \|_{0,\partial E\cap\Gamma_{\rm in}} \, \| \P0 w_h \|_{0,\partial E\cap\Gamma_{\rm in}} \\
& \gtrsim
- h^{-\frac{1}{2}} \, \| | \bb \cdot \nn^E |^{\frac{1}{2}} \P0 v_h \|_{0,\partial E\cap\Gamma_{\rm in}} \, \| \P0 w_h \|_{0,E} \\
& \gtrsim
- \| | \bb \cdot \nn^E |^{\frac{1}{2}} \P0 v_h \|_{0,\partial E\cap\Gamma_{\rm in}} \,  h^{\frac{1}{2}} \, \| \bb_h \cdot \nablah \Pg v_h \|_{0,\mathcal D(E)}\, .
\end{aligned}
\end{equation}
Gathering~\eqref{eq:eta1}--\eqref{eq:eta4} gives
\begin{equation}\label{eq:Nhinfsup}
\begin{aligned}
T_4 
\gtrsim 
 &- \Bigl( 
 \epsilon^{\frac{1}{2}} \| \nabla v_h \|_{0,E} 
+
\left(\dfrac{\epsilon}{\delta h}\right)^{\frac{1}{2}} \sum_{e\subset\partial E\cap\Gamma}\| \Pi_{k-1}^{0,e} v_h \|_{0,e} 
+
\| | \bb \cdot \nn^E |^{\frac{1}{2}} \P0 v_h \|_{0,\partial E\cap\Gamma_{\rm in}} 
\Bigr) \\
&\quad \cdot h^{\frac{1}{2}}  \| \bb_h \cdot \nablah \Pg v_h \|_{0,\mathcal D(E)} \, .
\end{aligned}
\end{equation}
{\it Estimate of} $\mathbf{(T_5).}$ 
The definition of the bilinear form $b_h(\cdot, \cdot)$ implies that
\begin{equation} \label{eq:b_split}
\begin{aligned}
T_5=b^E_h(v_h,w_h) 
&=
\bigl( \bb \cdot \nabla \P0 v_h, \, \P0 w_h \bigr)_{0,E} \\
& = 
\bigl( \bb \cdot \nabla \P0 v_h, \, w_h \bigr)_{0,E}  
+
\bigl( \bb \cdot \nabla \P0 v_h, \, ( \P0 - I ) w_h \bigr)_{0,E} \\
& =
\bigl( \bb \cdot \nabla \P0 v_h, \, h \bb_h \cdot \nabla \P0  v_h \bigr)_{0,E} \\
& \quad +        
\bigl( \bb \cdot \nabla \P0 v_h, \, w_h - h\bb_h \cdot \nabla \P0 v_h \bigr)_{0,E} \\
& \quad +        
\bigl(\bb\cdot \nabla\P0 v_h, \, (\P0 - I) w_h\bigr)_{0,E} \\
& \eqqcolon
\eta_{\bb_1} + \eta_{\bb_2} + \eta_{\bb_3} \, .
\end{aligned}
\end{equation}
For $\eta_{\bb_1}$, we add and subtract $(\bb \cdot \nabla \P0 v_h, h \bb\cdot \nabla \P0 v_h)_{0,E}$ and we use the Cauchy-Schwarz inequality to obtain
\begin{equation}\label{eq:bhinfsup1}
\begin{aligned}
\eta_{\bb_1} 
&=
(\bb \cdot \nabla \P0 v_h, \wh)_{0,E} \\
&=
h \, \| \bb \cdot \nabla \P0 v_h \|^2_{0,E}
+ 
(\bb \cdot \nabla \P0 v_h, h (\bb_h - \bb)\cdot \nabla \P0 v_h)_{0,E}\\
&\geq
h \, \| \bb \cdot \nabla \P0 v_h \|^2_{0,E}
- C\,h^{\frac{1}{2}} \| \bb \cdot \nabla \P0 v_h \|_{0,E} \, h^{\frac{1}{2}} | \bb |_{[W^{1,\infty}(E)]^d} h \| \nabla \P0 v_h \|_{0,E}\\
&\geq
h \, \| \bb \cdot \nabla \P0 v_h \|^2_{0,E}
- C\, h^{\frac{1}{2}} \| \bb \cdot \nabla \P0 v_h \|_{0,E} \, h^{\frac{1}{2}} | \bb |_{[W^{1,\infty}(E)]^d} \| v_h \|_{0,E}\, ,
\end{aligned}
\end{equation}
where in the third step we have used $\|\bb_h-\bb\|_{[L^\infty(E)]^d}\lesssim h| \bb |_{[W^{1,\infty}(E)]^d}$. \\
For $\eta_{\bb_2}$, recalling the definition of $w_h$ and using the Young inequality to split the two terms, we get
\begin{equation}\label{eq:adv-eq-xi}
\begin{aligned}
\eta_{\bb_2} & = h \bigl ( \bb\cdot \nabla \P0 v_h, (\pi - I) (\bb_h \cdot \nabla \P0  v_h) \bigr)_{0,E} \\
& \geq
- \dfrac{h}{2} \| \bb\cdot \nabla \P0 v_h \|^2_{0,E} 
- \dfrac{h}{2} \| (\pi - I) (\bb_h \cdot \nabla \P0 v_h) \|^2_{0,E} \, .
\end{aligned}
\end{equation}
For the second term on the right-hand side, we use the fact that $\bb_h \cdot \nabla \Pi^0_k v_h \in \mathbb P_k(\Omega_h)$. Theorem~\ref{prp:clmest} gives
\begin{equation}\label{eq:xi-sum}
h \| (\pi - I) (\bb_h \cdot \nabla \Pg v_h) \|_{0,E}^2 
\lesssim
h^2 \! \sum_{e \subset \partial E\setminus\Gamma} \| [\![\bb_h \cdot \nablah \Pg v_h]\!] \|_{0,e}^2 
\, .
\end{equation}
The triangular inequality and the definition of the jump bilinear form give
\begin{equation}\label{eq:xi-eq}
\begin{aligned}
h^2 \! \sum_{e \subset \partial E\setminus\Gamma} \| [\![\bb_h \cdot \nablah \Pg v_h]\!] \|_{0,e}^2 
& \lesssim 
h^2 \! \sum_{e \subset \partial E\setminus\Gamma} \| [\![( \bb_h - \bb ) \cdot \nablah \Pg v_h]\!] \|_{0,e}^2
+
h^2 \! \sum_{e \subset \partial E\setminus\Gamma} \| [\![\bb \cdot \nablah \Pg v_h]\!] \|_{0,e}^2 \\
& \lesssim 
h^2 \! \sum_{e \subset \partial E\setminus\Gamma} \| [\![( \bb_h - \bb ) \cdot \nablah \Pg v_h]\!] \|_{0,e}^2
+
h^2 \! \sum_{e \subset \partial E\setminus\Gamma} \gamma_e^2  \| [\![ \nablah \Pg v_h]\!] \|_{0,e}^2 \\
& \lesssim 
h^2 \! \sum_{e \subset \partial E\setminus\Gamma} \| [\![( \bb_h - \bb ) \cdot \nablah \Pg v_h]\!] \|_{0,e}^2
+
\sum_{E'\subset \mathcal{D}(E)}J_h^{E'} (v_h, v_h)\ ,
\end{aligned}
\end{equation}
where we used that $\gamma^2_e\le \gamma_e$ (since $\gamma_e\le 1$, see \eqref{eq:beta-scaling}).
On each $e$, the argument in the first sum on the right-hand side of the previous inequality is controlled using the trace inequality and standard estimates on $\bb \in [W^{1,\infty}(\Omega)]^d$:
\begin{equation}
\label{eq:xi-eq-2}
\begin{aligned}
h^2\| [\![(\bb_h -\bb)\cdot \nablah \Pg v_h]\!] \|_{0,e}^2
& \lesssim 
h^4  | \bb |^2_{[W^{1,\infty}(E\cup E')]^d} h^{-1}\|  \nablah \Pg v_h  \|_{0,E\cup E'}^2  \\
& \lesssim
h  | \bb |^2_{[W^{1,\infty}(E\cup E')]^d} \|  \Pg v_h  \|_{0,E\cup E'}^2  \\
& \lesssim
h  | \bb |^2_{[W^{1,\infty}(E\cup E')]^d} \| v_h  \|_{0,E\cup E'}^2\, ,
\end{aligned}
\end{equation}
where $E$ and $E'$ are the two elements sharing the edge $e$.
Combining \eqref{eq:adv-eq-xi} with \eqref{eq:xi-eq-2}, we obtain for $\eta_{\bb_2}$
\begin{equation} \label{eq:bhinfsup2}
\begin{aligned}
\eta_{\bb_2} 
 \geq & - 
\dfrac{h}{2} \| \bb \cdot \nabla \P0 v_h \|^2_{0,E}\\
&- C\Big(
h | \bb |^2_{[W^{1,\infty}(\mathcal{D}(E))]^d} \| v_h \|_{0,\mathcal{D}(E)}^2
+\sum_{E'\subset \mathcal{D}(E)}J_h^{E'} (v_h, v_h)
\Big) \, .
\end{aligned}
\end{equation}
It remains to control $\eta_{\bb_3}$. Since $\bb_h\in [\Pk_1(E)]^d$, it holds $\bigl( \bb_h \cdot \nabla \P0 v_h, (\P0 - I) w_h \bigr)_{0,E}= 0$. 
Hence we have 
\begin{equation} \label{eq:bhinfsup3}
\begin{aligned}
\eta_{\bb_3} & = \bigl( (\bb - \bb_h) \cdot \nabla \P0 v_h, (\P0 - I) w_h \bigr)_{0,E} \\ 
& \gtrsim 
- \| (\bb - \bb_h) \cdot \nabla \P0 v_h \|_{0,E} \, \| h\pi(\bb_h\cdot \nabla\P0 v_h) \|_{0,E} \\
& \gtrsim 
- | \bb |_{[W^{1,\infty}(E)]^d} h \| \nabla \P0 v_h \|_{0,E} \, h \| \bb_h \cdot \nablah \Pg v_h \|_{0,\mathcal{D}(E)}  \\
& \gtrsim
-    | \bb |_{[W^{1,\infty}(E)]^d} \| v_h \|_{0.\mathcal D(E)}^2 \, . 
\end{aligned}
\end{equation}
Collecting \eqref{eq:bhinfsup1}, \eqref{eq:bhinfsup2} and \eqref{eq:bhinfsup3}, from \eqref{eq:b_split} we get 
\begin{equation}\label{eq:advt-term-est-E}
\begin{aligned}
T_5 
\geq
\dfrac{h}{2} \, & \| \bb \cdot \nabla \P0 v_h \|^2_{0,E} 
- C\Big(
\sum_{E'\subset \mathcal{D}(E)}J_h^{E'} (v_h, v_h) \\
&+ h^{\frac{1}{2}} \| \bb \cdot \nabla \P0 v_h \|_{0,E} \, h^{\frac{1}{2}} | \bb |_{[W^{1,\infty}(E)]^d} \| v_h \|_{0,E}  \\
&    
+ h | \bb |^2_{[W^{1,\infty}(\mathcal{D}(E))]^d} \| v_h \|_{0,\mathcal{D}(E)}^2
+    | \bb |_{[W^{1,\infty}(E)]^d}  \| v_h \|_{0,\mathcal{D}(E)}^2 
\Big) \, .
\end{aligned}
\end{equation}
{\it Estimate of} $\mathbf{(T_6).}$ 
The last term that we have to estimate is related to $d_h(\cdot,\cdot)$. We use the Cauchy-Schwarz and the trace inequality for polynomials: 
\begin{equation}
\begin{aligned}
T_6&=d^E_h(v_h, w_h) 
= 
-\dfrac{1}{2} 
\sum_{e \subset \partial E \setminus \Gamma} \int_e \bb \cdot [\![\Pg v_h]\!] \{\Pg w_h\} {\rm d}s \\
& \gtrsim
- \sum_{e \subset \partial E \setminus \Gamma}
\| \bb \cdot [\![\Pg v_h]\!] \|_{0,e}
\| \{\Pg w_h\} \|_{0,e} \\
& \gtrsim
- \sum_{e \subset \partial E \setminus \Gamma}
\left(\| | \bb \cdot \nn^E |^{\frac{1}{2}} \P0 v_h \|_{0,e}
+\| | \bb \cdot \nn^{E'} |^{\frac{1}{2}} \Pi_k^{0,E'} v_h \|_{0,e}
\right)
h^{-\frac{1}{2}} \| w_h\|_{0,E \cup E'} \\
& \gtrsim
- \sum_{e \subset \partial E \setminus \Gamma}
\left(\| | \bb \cdot \nn^E |^{\frac{1}{2}} \P0 v_h \|_{0,e}
+\| | \bb \cdot \nn^{E'} |^{\frac{1}{2}} \Pi_k^{0,E'} v_h \|_{0,e} \right) h^{\frac{1}{2}} \| \bb_h\cdot\nablah \Pg v_h\|_{0,\mathcal{D}(E \cup E')}
 \, .
\end{aligned}
\end{equation}

\noindent
By collecting the estimates of all six terms and adding over all elements, we obtain
\begin{equation}\label{eq:almost-done}
\begin{aligned}
\Acip & (v_h, w_h) 
\geq
\frac{h}{2} \| \bb \cdot \nablah \Pg v_h \|_{0}^2
- C \Big(
\sum_{E \in \Omega_h} \big( \epsilon^{\frac{1}{2}} \| \nabla v_h \|_{0,E} +  J_h^E(v_h, v_h)^{\frac{1}{2}}  +  \| v_h \|_{0,E}  \\
& 
+
\left(\dfrac{\epsilon}{\delta h}\right)^{\frac{1}{2}} \sum_{e\subset\partial E\cap\Gamma}\| \Pi_{k-1}^{0,e} v_h \|_{0,e}  
+
\| | \bb \cdot \nn^E |^{\frac{1}{2}} \P0 v_h \|_{0,\partial E\cap\Gamma_{\rm in}} \big)    h^{\frac{1}{2}} \| \bb_h \cdot \nabla \P0 v_h \|_{0,E} \\
& 
+  J_h(v_h, v_h)   +
\sum_{E \in \Omega_h} \left(h | \bb |^2_{[W^{1,\infty}(E)]^d}+    | \bb |_{[W^{1,\infty}(E)]^d}\right) \| v_h \|_{0,E}^2 \\
& + 
\sum_{E \in \Omega_h}  h^{\frac{1}{2}} \| \bb_h \cdot \nabla \P0 v_h \|_{0,E} \, h^{\frac{1}{2}} | \bb |_{[W^{1,\infty}(E)]^d} \| v_h \|_{0,E}
\Big)\, .  
\end{aligned}
\end{equation}
Above, we have also used the property that, due to assumption \textbf{(A1)}, summing over the elements, each element is counted only a uniformly bounded number of times, even when the terms involve norms on $\mathcal{D}(E)$ or $\mathcal{D}(\mathcal{D}(E))$.

We now notice that the triangular inequality, standard approximation results, and an inverse estimate give
\begin{equation}\label{falcao}
h^{\frac{1}{2}} \| \bb_h \cdot \nabla \P0 v_h \|_{0,E}
\lesssim
h^{\frac{1}{2}}  \left( \| \bb \cdot \nabla \P0 v_h \|_{0,E}
+
| \bb |_{[W^{1,\infty}(E)]^d} \| v_h \|_{0,E} \right).
\end{equation}
Hence, from \eqref{eq:almost-done}, using~\eqref{falcao} and the Young inequality (with suitable constants) for the first and the last summations on the right-hand side, we get 
\begin{equation*}
    \begin{aligned}
\Acip(v_h, w_h) 
\geq  &C_1\, h \| \bb \cdot \nablah \Pg v_h \|_{0}^2 - C_2\,\Big(\epsilon \| \nablah v_h \|_{0}^2  +   \sigma\| v_h \|_{0}^2 + \sum_{E\in\Omega_h}\dfrac{\epsilon}{\delta h_E} \sum_{e\subset\partial E\cap \Gamma} 
\|\Pi_{k-1}^{0,e}v_h\|_{0,e}^2\\
& \qquad\quad
+\sum_{E\in\Omega_h} \||\bb\cdot\nn^E|^{\frac{1}{2}}\P0 v_h\|_{0,\partial E\cap \Gamma_{\text{in}}}^2
+ J_h(v_h, v_h)\Big) 
\, .  
\end{aligned}
\end{equation*}
From Lemma \ref{lm:symm-part}, we now obtain
\[
\Acip(v_h, w_h) 
\geq
C_1\, h \| \bb \cdot \nablah \Pg v_h \|^2_0 
-C_2\, \, \Acip(v_h, v_h) \, .
\]
\end{proof}

Whenever $h>\varepsilon$, we also need the following estimate.
\begin{lemma}\label{lm:whvh}
Assume that $h>\varepsilon$. For any $v_h \in V^{k,nc}(\Omega_h)$, let $w_h$ be defined as in~\eqref{eq:wh}. Then,
\[
\| w_h \|_\cip \lesssim\| v_h \|_\cip\, ,
\]
with hidden constant independent of $\epsilon$, $h$, and~$v_h$.
\end{lemma}
\begin{proof}
From the definition of $w_h$ and $h>\varepsilon$, for the gradient term in the $\|\cdot\|_{\cip}$ norm, we have
\begin{equation}\label{eq:cont-a}
\epsilon \| \nabla w_h \|^2_{0,E} 
\lesssim \epsilon h^{-2}\|w_h\|_{0,E}
\lesssim
\epsilon \| \pi (\bb_h \cdot \nabla\P0 v_h) \|^2_{0,E}
\lesssim
h \| \bb_h \cdot \nablah \Pg v_h \|^2_{0,\mathcal D(E)} \, .
\end{equation}
For the $L^2$ term, we have that
\begin{equation}\label{eq:cont-c}
\sigma \| w_h \|^2_{0,E}
\lesssim
\sigma \| h \pi (\bb_h \cdot \nabla\P0 v_h)\|_{0,E}
\lesssim
\sigma
\| v_h \|^2_{0,\mathcal{D}(E)} \, .
\end{equation}
The convective term in the norm is estimated by
\begin{equation}\label{eq:cont-b}
h \| \bb \cdot \nabla \P0 w_h \|_{0,E}^2 
\lesssim 
  h^{-1} \| \P0 w_h \|_{0,E}^2
\lesssim
  h \| \bb_h \cdot \nablah \Pg v_h \|_{0,\mathcal{D}(E)}^2 \, .   
\end{equation}
The boundary terms are estimated by
\begin{equation}\label{eq:cont-Na}
\sum_{e \subset \partial E \cap \Gamma}\dfrac{\epsilon}{\delta h} \| \Pi^{0,e}_{k-1} w_h \|^2_{0,e}
\leq
\dfrac{\epsilon}{\delta h} \| w_h \|^2_{0,\partial E\cap\Gamma}
\lesssim
\epsilon
\| \pi (\bb_h \cdot \nabla \P0 v_h) \|^2_{0,E}
\lesssim
h
\| \bb_h \cdot \nablah \Pg v_h \|^2_{0,\mathcal D(E)} \, ,
\end{equation}
and
\begin{equation}\label{eq:cont-Nb}
\| | \bb \cdot \nn^E|^{\frac{1}{2}} \P0 w_h \|^2_{0,\partial E\cap\Gamma_{\rm in}} 
\lesssim
h^{-1} \| \P0 w_h \|^2_{0,E}
\lesssim
h \| \bb_h \cdot \nablah \Pg v_h \|_{0,\mathcal{D}(E)}^2 \, .
\end{equation}
We gather~\eqref{eq:cont-a}, \eqref{eq:cont-c}, \eqref{eq:cont-b}, \eqref{eq:cont-Na}, \eqref{eq:cont-Nb}, together with estimate~\eqref{eq:Jhinfsup} for the term~$J_h^E(w_h,w_h)$, add over all elements, and obtain
\[
\| w_h \|_{\cip}^2\lesssim
\sigma
\| v_h \|^2_{0}
+\sum_{E\in \Omega_h}h\| \bb_h \cdot \nablah \Pg v_h \|_{0,E}^2\, ,
\]
and the result follows from estimate~\eqref{falcao}.
\end{proof}

We are now able to prove the inf-sup condition~\eqref{eq:infsup}.
\begin{theorem}[Inf-sup condition]\label{thm:inf-sup}
Under assumptions \textbf{(A1)}, 
\begin{equation}\label{eq:inf-sup-cond} 
\| v_h \|_{\cip} 
\lesssim 
\sup_{z_h \in V_h^{k,nc}(\Omega_h)} \dfrac{\Acip (v_h, z_h)}{\| z_h \|_{\cip}}
\qquad \text{for all $v_h \in V^{k,nc}_h(\Omega_h)$.}
\end{equation}
\end{theorem}

\begin{proof} 
We distinguish two cases: $h>\epsilon$ and $h\le \epsilon$. 

If $h>\epsilon$, given $v_h \in V^{k,nc}(\Omega_h)$, we define the function
$z_h \coloneqq w_h + \kappa v_h$,
where $w_h$ is given by~\eqref{eq:wh}, and $\kappa$ is a sufficiently large constant. Combining Lemma \ref{lm:symm-part} with Lemma \ref{lm:adv-term} gives
\[
\Acip(v_h, z_h) =
\Acip(v_h, w_h + \kappa v_h) 
\gtrsim 
\| v_h \|_\cip^2 \, .
\]
This, together with Lemma~\ref{lm:whvh}, gives the inf-sup condition in the case $h>\epsilon$.\\
If $h\le \epsilon$, from Lemma~\ref{lm:symm-part}, the definition of the norm $\|\cdot\|_{\rm cip}$, and the estimate
$$
h  \| \bb \cdot \nabla \P0 v_h \|^2_{0,E} 
\lesssim \varepsilon \| \nabla \P0 v_h \|^2_{0,E}
\lesssim \varepsilon \| \nabla v_h \|^2_{0,E} \, ,
$$
we obtain
$\Acip (v_h, v_h)\gtrsim \|v_h\|_{\rm cip}^2$
and the proof is complete.
\end{proof}

\subsection{Error Estimates}

In this section, we derive error estimates under the following smoothness assumption on the terms appearing in~\eqref{eq:strong}. 

\medskip\noindent
\textbf{(A2) Smoothness 
assumption.} The solution of the continuous problem $u$, the right-hand side $f$, and the advective field $\bb$ in~\eqref{eq:weak-2} satisfy
\[
\begin{aligned}
&u \in  H^2(\Omega)\cap H^1_g(\Omega)\cap H^{k+1}(\Omega_h)\,, &
&f \in H^{k+\frac{1}{2}}(\Omega_h)\,, &
&\bb \in [W^{k+1}_{\infty}(\Omega_h)]^d\, . &
\end{aligned}
\]

\medskip

Let  $u_h$ be the discrete solution of \eqref{eq:cip-vem}. 
Then, 
thanks to the inf-sup condition \eqref{eq:inf-sup-cond}, for the error~$u-u_h$, we prove the following result.
\begin{proposition}
\label{prp:abstract}
Under assumptions \textbf{(A1)}, 
we have
\begin{equation}
\label{eq:abstract}
\| u - u_h \|_{\cip} 
\lesssim 
\| e_\mathcal{I} \|_{\cip}
+ \errF +  \errBt + \erra + \errb + \errc + \eta_d + \errN + \errJ  \, ,
\end{equation}
where $e_\mathcal{I} := u - u_\mathcal{I}$ and $\uint \in V_h^{k,nc}(\Omega_h)$ is the interpolant function of $u$ defined in Lemma \ref{lm:interpolation}.
Moreover, in the right-hand side of \eqref{eq:abstract} we have defined
\begin{equation}\label{eq:interpol_errors}
\begin{aligned}
\errF &\coloneqq  \| \tilde{\mathcal{F}} - \mathcal{F}_h \|_\cipdual \, ,
\\
\errBt &\coloneqq \| \tilde{\mathcal B}(u,\cdot)  \|_\cipdual \, ,
\\
\erra &\coloneqq \epsilon \, \| a(u, \cdot) - a_h(\uint, \cdot) \|_\cipdual \,,
\\
\errb &\coloneqq \|b(u, \cdot) - b_h(\uint, \cdot) \|_\cipdual \, ,
\\
\errc &\coloneqq   \sigma \| c(u, \cdot) -  c_h(\uint, \cdot) \|_\cipdual \,,
\\
\eta_d &\coloneqq \| d_h(\uint, \cdot) \|_\cipdual \, ,
\\
\errN &\coloneqq \| \Tilde{\mathcal{N}}_h(u, \cdot) - \mathcal{N}_h(\uint, \cdot) \|_\cipdual \, ,
\\
\errJ &\coloneqq  \|\tilde{J}_h(u ,   \cdot) - J_h(\uint,   \cdot)\|_\cipdual  = \|J_h(\uint,   \cdot)\|_\cipdual\, , \\
\end{aligned}
\end{equation}
where $\|\cdot \|_\cipdual$ denotes the dual norm of the norm $\|\cdot \|_\cip$. 
\end{proposition}

\begin{proof}
Setting $e_h := u_h - u_\mathcal{I}$,
thanks to the triangular inequality, we have 
\[
\| u - u_h \|_\cip \leq \| e_{\mathcal{I}} \|_\cip + \| e_h \|_\cip \, .
\]
Hence, we only have to bound the second term of the right-hand side. Using the inf-sup condition, \eqref{eq:cip-vem} and \eqref{eq:consistency-1}, we have that
\begin{equation}\label{eq:error-e_h}
\begin{aligned}
\| e_h \|_{\cip} 
& =
\sup_{v_h \in V_h^{k,nc}(\Omega_h)} \dfrac{\Acip(u_h - u_\mathcal{I}, v_h)} {\| v_h \|_{\cip}}
=
\sup_{v_h \in V_h^{k,nc}(\Omega_h)} \dfrac{\mathcal{F}_h(v_h) - \Acip(u_\mathcal{I}, v_h)} {\| v_h \|_{\cip}}
\\
& =
\sup_{v_h \in V_h^{k,nc}(\Omega_h)} \dfrac{\mathcal{F}_h(v_h) - \tilde{\mathcal{F}}(v_h) - \tilde{\mathcal{B}}(u, v_h) + \Atcip(u, v_h) - \Acip(u_\mathcal{I}, v_h)} {\| v_h \|_{\cip}}.
\end{aligned}
\end{equation}
Estimate \eqref{eq:abstract} now follows from considering the definitions of~$\Acip(\cdot, \cdot)$ and~$\Atcip(\cdot, \cdot)$ given in \eqref{eq:Acip} and \eqref{eq:Atcip}, respectively.
\end{proof}

We proceed by estimating each of the terms on the right-hand side of~\eqref{eq:abstract}. We start with
the interpolation error in the CIP norm.
\begin{lemma}\label{lm:est_eI}
Under assumptions \textbf{(A1)} and \textbf{(A2)}, we have the following estimate:
\begin{equation}\label{eq:eI}
\| \eint \|^2_{\cip}
\lesssim
\sum_{E\in \Omega_h}
\epsilon  \, h^{2\reg} \, \vert u \vert^2_{\reg+1,E} +  \sum_{E\in \Omega_h} h^{2\reg+1} \, \vert u \vert^2_{\reg+1,E} \, .
\end{equation}
\end{lemma}

\begin{proof}
The proof can be developed along the lines of \cite{BLT:2024}. The only slight differences lie in the treatment of the following Nitsche terms, for which we use a trace inequality and the 
interpolation estimate in Lemma~\ref{lm:interpolation} to obtain
\[
\sum_{e \subset \partial E \cap \Gamma} \dfrac{\epsilon}{\delta h_E} \langle  \Pi^{0,e}_{k-1} e_\mathcal{I},  \Pi^{0,e}_{k-1} e_\mathcal{I} \rangle_{e}
\lesssim
\sum_{e \subset \partial E \cap \Gamma} \dfrac{\epsilon}{\delta h}
\| e_{\mathcal I} \|_{0,e}^2
\lesssim 
\dfrac{\epsilon}{\delta h^2} \| e_\mathcal{I} \|_{0,E}^2
+
\dfrac{\epsilon}{\delta} | e_\mathcal{I} |_{1,E}^2
\lesssim 
\epsilon \, h^{2\reg} \, |u|_{\reg+1,E}^2 \, ,
\]
and 
\[
\sum_{e \subset \partial E \cap \Gamma_{\text{in}}} \langle | \bb \cdot \nn^E |  \P0 e_\mathcal{I},  \P0 e_\mathcal{I} \rangle_{e}
\lesssim
\, h^{-1} \, \| \P0 e_\mathcal{I}\|_{0,E}^2
\leq
\, h^{-1} \, \| e_\mathcal{I}\|_{0,E}^2
\lesssim
{\color{blue}   } \, h^{2\reg+1} \, |u|_{\reg+1,E}^2 \, .
\]
\end{proof}

Now, we estimate the
term~$\errBt$ defined in~\eqref{eq:interpol_errors},
which enters into play because of the nonconformity of our method.
\begin{lemma}[Estimate of $\errBt$]
\label{lm:errB}
Under assumptions \textbf{(A1)} and \textbf{(A2)}, we have the following estimate:
\begin{equation}\label{eq:Best}
\errBt
\lesssim
\epsilon^{\frac12} h^{\reg}
\Big(\sum_{E\in\Omega_h} | u |_{\reg+1,E}^2\Big)^{1/2} \, .
\end{equation}	
\end{lemma}

\begin{proof}
Thanks to the definition of the space $V_h^{k,nc}(\Omega_h)$ and Lemma \ref{lm:bh-edge}, 
we get
\begin{equation}\label{eq:BtildeEst}
\begin{aligned}
\tilde{\mathcal{B}}(u,v_h) 
&=
\sum_{e \in {\mathcal E}_h^0} \epsilon \int_e \nabla u\cdot [\![v_h]\!] {\rm d}s 
 = 
\sum_{e \in {\mathcal E}_h^0}  \epsilon \int_e
({\mathbf I} - {\mathbf \Pi}^{0,e}_{k-1})  \nabla u \cdot [\![v_h]\!] {\rm d}s \\
&= 
\sum_{e \in {\mathcal E}_h^0} \epsilon  \int_e
({\mathbf I} - {\mathbf \Pi}^{0,e}_{k-1})  \nabla u \cdot([\![v_h]\!] -  [\![\Pi^{0,e}_0 v_h]\!]) {\rm d}s \\
& 
\leq
\sum_{e \in {\mathcal E}_h^0}  \epsilon^{\frac{1}{2}} \| ({\mathbf I} - {\mathbf \Pi}^{0,e}_{k-1})  \nabla u \|_{0,e} \, \epsilon^{\frac{1}{2}} \| \, [\![v_h]\!] -  [\![\Pi^{0,e}_0 v_h]\!] \, \|_{0,e}  \\
& \leq
\sum_{E \in \Omega_h}  \epsilon^{\frac{1}{2}} \, h^{k-\frac{1}{2}} \, |u|_{k+1,E} \, \epsilon^{\frac{1}{2}} \, h^{\frac{1}{2}} |v_h|_{1,E}  \\
&\lesssim 
\epsilon^{\frac{1}{2}} h^\reg 
\sum_{E \in \Omega_h}
| u |_{\reg+1,E} \| v_h \|_{\cip,E} \\
&\lesssim 
\epsilon^{\frac{1}{2}} h^\reg 
\Big(\sum_{E \in \Omega_h}
| u |_{\reg+1,E}^2\Big)^{1/2} \| v_h \|_{\cip} \, ,
\end{aligned}
\end{equation}
where we have used that each element is counted a finite number of times. 
Estimate~\eqref{eq:Best} follows from \eqref{eq:BtildeEst}.
\end{proof}

Some of the other quantities in~\eqref{eq:interpol_errors}
can be estimated exactly as in~\cite{BLT:2024}. We group them in the following lemma. 

\begin{lemma}[Estimates of $\erra$, $\errc$, $\errJ$ and $\errF$]
\label{lm:SameAsConf}
Under assumptions \textbf{(A1)} and \textbf{(A2)}, we have the following estimates:
\begin{equation}
\label{eq:SameAsConf}
\begin{aligned}
& \erra 
\lesssim \epsilon^{\frac12} \, h^{\reg} 
\Big(\sum_{E \in \Omega_h} \, \vert u \vert_{\reg+1,\mathcal D(E)}^2\Big)^{1/2} \, , \\
& \errc 
\lesssim h^{\reg + 1} 
\Big(\sum_{E \in \Omega_h}  \vert u \vert_{\reg+1,E}^2\Big)^{1/2} \, ,   \\ 
& \errJ
\lesssim h^{\reg + \frac12} 
\Big(\sum_{E \in \Omega_h}  | u |_{\reg+1,E}^2 \Big)^{1/2} \, , \\
& \errF \lesssim h^{\reg + \frac12} \, 
\Big(\sum_{E \in \Omega_h}  \vert f \vert_{\reg + \frac{1}{2},E}^2\Big)^{1/2} \, .
\end{aligned}
\end{equation}	
\end{lemma}

\begin{proof}
We refer to~\cite[Sect.~3.3]{BLT:2024}. We observe that, for~$\errF$, the terms containing~$g$ cancel with each other (cf.~\eqref{eq:FEh} and~\eqref{eq:tildeFEh}), therefore the estimate contains only~$f$.
\end{proof}

In the following three lemmas, we 
detail the proofs of the estimates for 
the remaining terms in~\eqref{eq:interpol_errors}.

\begin{lemma}[Estimate of $\errb$]
\label{lm:errb}
Under assumptions \textbf{(A1)} and \textbf{(A2)}, the term $\errb$ satisfies
\begin{equation}\label{eq:etab_1}
\errb 
\lesssim h^{\reg} 
\Big(\sum_{E\in\Omega_h} \| \bb \|_{[W^{\reg}_\infty(E)]^2}^2 | u |_{\reg+1,E}^2 \Big)^{1/2}
\, .
\end{equation}			
\end{lemma}

\begin{proof}
We have
\begin{equation*}
\errb
= \sup_{v_h \in V_h^{k,nc}(\Omega_h)}
\dfrac{\sum_{E \in \Omega_h}( \bb \cdot \nabla u, v_h)_{0,E}  - (\bb \cdot \nabla \P0 \uint, \P0 v_h)_{0,E}}
{\| v_h \|_{\cip}} \, .
\end{equation*}
We proceed locally element by element. For any fixed element $E\in \Omega_h$, we have 
\[
\begin{aligned}
( \bb \cdot \nabla u,& v_h)_{0,E} 
- 
(\bb \cdot \nabla \P0 \uint, \P0 v_h)_{0,E} \\
& = 
(\bb \cdot \nabla (u - \uint), v_h)_{0,E}
+
((I - \P0)(\bb \cdot \nabla \uint), (I - \P0) v_h)_{0,E} \\
& \quad +
( \bb \cdot \nabla (I- \P0) \uint, \P0 v_h)_{0,E} \\
&\eqqcolon
T^E_{b,1} + T^E_{b,2} + T^E_{b,3}
\, .
\end{aligned}
\]
We consider each of the three terms. On the first one, we use the Cauchy-Schwarz inequality, the interpolation estimate, and the definition of $\| \cdot \|_{\cip,E}$, and obtain
\[
T^E_{b,1} = (\bb \cdot \nabla (u - \uint), v_h)
\lesssim \| \bb \|_{[L^\infty(E)]^d}
| u - \uint |_{1,E} \| v_h \|_{0,E}
\lesssim h^\reg \| \bb \|_{[L^\infty(E)]^d} |u|_{\reg +1,E} \| v_h \|_{\cip,E} \, .
\]
For the second one, we have
\[
\begin{aligned}
T^E_{b,2} 
&=
((I - \P0)(\bb \cdot \nabla \uint), (I - \P0) v_h)_{0,E} \\
&=
((I - \P0)(\bb \cdot \nabla u), (I - \P0) v_h)_{0,E}
+
((I - \P0)(\bb \cdot \nabla (\uint - u)), (I - \P0) v_h)_{0,E} \\
& \lesssim
\bigl(
\| (I - \P0)(\bb \cdot \nabla u) \|_{0,E} + \|(I - \P0)(\bb \cdot \nabla (\uint - u))\|_{0,E} 
\bigr)
\| (I - \P0) v_h \|_{0,E} \\
& \lesssim
\bigl(
h^\reg | \bb \cdot \nabla u |_{\reg,E} + \| \bb \|_{[L^\infty(E)]^d} | u - \uint|_{1,E} 
\bigr)
\| v_h \|_{\cip,E} \\
& \lesssim
h^\reg \| \bb \|_{[W^{\reg,\infty(E)}]^d}\| u \|_{\reg +1,E} \| v_h \|_{\cip,E} \, .
\end{aligned}
\]
The third term can be estimated as
\[
\begin{aligned}
T^E_{b,3} 
&=
( \bb \cdot \nabla (I- \P0) \uint, \P0 v_h)_{0,E} \\
&\lesssim
\| \bb \|_{[L^\infty(E)]^d}| (I - \P0) \uint |_{1,E} \| \P0 v_h \|_{0,E} \\
& \lesssim
\| \bb \|_{[L^\infty(E)]^d} \bigl(| (I - \P0) u |_{1,E} +  | (I - \P0) (u - \uint) |_{1,E} \bigr)\| \P0 v_h \|_{0,E} \\
& \lesssim
h^\reg \| \bb \|_{[L^\infty(E)]^d} | u |_{\reg+1,E} \| v_h \|_{\cip,E} \, .
\end{aligned}
\]
Estimate \eqref{eq:etab_1} 
follows from the above three bounds, and from summing over all the elements $E \in \Omega_h$.
\end{proof}

\begin{lemma}[Estimate of $\eta_d$]
\label{lm:errd}
Under assumptions \textbf{(A1)} and \textbf{(A2)},
the term $\eta_d$ can be estimated as follows:
\begin{equation}\label{eq:errd}
\eta_d
\lesssim h^{\reg} 
\Big(\sum_{E \in \Omega_h} |u|_{\reg+1,\mathcal{D}(E)}^2\Big)^{1/2} \, .
\end{equation}
\end{lemma}

\begin{proof}
We have
\begin{equation*}
\eta_d = 
\sup_{v_h \in V_h^{k,nc}(\Omega_h)}
\frac{ -\frac{1}{2}\sum_{E\in\Omega_h}\sum_{e \subset \partial E / \Gamma} \int_e \bb\cdot [\![\Pi_k^0 \uint]\!]
\{ \Pi_k^0 v_h \} {\rm d}s}
{\| v_h \|_{\cip}}\,.
\end{equation*}
Again, we proceed element by element.
Using also a trace inequality for polynomials, we estimate the numerator of the quotient above as follows: 
\[
\begin{aligned}
& -  \sum_{e \subset \partial E \setminus \Gamma}  \dfrac{1}{2}  \int_e \bb\cdot [\![\Pi_k^0 \uint]\!]
\{ \Pi_k^0 v_h \} {\rm d}s  = 
\sum_{e \subset \partial E \setminus \Gamma} \dfrac{1}{2} \int_e  \bb\cdot [\![u - \Pi_k^0 \uint]\!]
\{ \Pi_k^0 v_h \} {\rm d}s \\
& = 
\sum_{e \subset \partial E \setminus \Gamma} \dfrac{1}{2}  \int_e \bb\cdot [\![u - \Pi_k^0 u]\!]
\{ \Pi_k^0 v_h \} {\rm d}s
+
\sum_{e \subset \partial E \setminus \Gamma} \dfrac{1}{2} \int_e \bb \cdot [\![\Pi^0_k (u - \uint )]\!] \{\Pi^{0}_k v_h\} {\rm d}s
\\
& \lesssim
\| \bb \|_{[L^\infty(\mathcal{D}(E))]^d} \sum_{E' \subset \mathcal{D}(E) }  \bigl( h^{\frac{1}{2}} |  u - \Pi^0_k u |_{1,E'} +
h^{-\frac{1}{2}} \|  u - \Pi^0_k u \|_{0,E'}\\
& \qquad \qquad + h^{-\frac{1}{2}} \| \Pi^0_k (u - \uint ) \|_{0,E'} \bigr) h^{-\frac{1}{2}} \| v_h\|_{0,\mathcal{D}(E)} \\
& \lesssim
h^{\reg} |u|_{\reg +1,\mathcal{D}(E)} \| v_h \|_{\cip,\mathcal{D}(E)} \, .
\end{aligned}
\]
Estimate \eqref{eq:errd} follows from the bound above, by summing over the elements and taking into account, as already noticed for \eqref{eq:almost-done}, that each element is counted only a uniformly bounded number of times.
\end{proof}

\begin{lemma}[Estimate of $\errN$]
\label{lm:errN}
Under assumptions \textbf{(A1)} and \textbf{(A2)},
the term $\errN$ can be estimated as follows:
\begin{equation}\label{eq:errn}
\errN
\lesssim 
(\epsilon^\frac12  h^{\reg}  +  h^{\reg + \frac12} ) 
\Big(\sum_{E \in \Omega_h} 
 | u |_{\reg +1,E}^2 \Big)^{1/2} \, .
\end{equation}
\end{lemma}

\begin{proof}
We have
\begin{equation*}
\begin{aligned}
\errN 
= &\sup_{v_h \in V_h^{k,nc}(\Omega_h)}  \Big[ \frac{1}{\| v_h \|_{\cip}} \sum_{E \in \Omega_h}
\Big(
\sum_{e \subset \partial E\cap\Gamma} \epsilon \langle \PZ0P \nabla  \uint \cdot \nn^{E} - \nabla u \cdot \nn^{E} , v_h \rangle_{e}\\
&\quad+ 
\sum_{e \subset \partial E\cap\Gamma} \epsilon \langle \uint - u , \PZ0P \nabla  v_h \cdot \nn^{E} \rangle_{e}  \\ 
&\quad 
+ \dfrac{\epsilon}{\delta h_E}\sum_{e \subset \partial E\cap\Gamma}    \langle  u - \uint, \Pi^{0,e}_{k-1} v_h \rangle_{e} + \sum_{e \subset \partial E\cap\Gamma_{\text{in}}}
\langle | \bb \cdot \nn^{E} | ( u - \P0 \uint ), \P0 v_h\rangle_{e}\Big)\Big]\, .
\end{aligned}
\end{equation*}
We consider an element $E \in \Omega_h$.
Thus, we need to estimate four different terms:
\begin{equation}
\begin{aligned}
\sum_{e \subset \partial E \cap \Gamma}\epsilon \langle & \PZ0P \nabla  \uint \cdot \nn^E - \nabla u \cdot \nn^E , v_h \rangle_{e}
+ 
\sum_{e \subset \partial E \cap \Gamma} \epsilon \langle \uint - u , \PZ0P \nabla  v_h \cdot \nn^{E} \rangle_{e}  \\ 
&\quad 
+ \dfrac{\epsilon}{\delta h_E}  \sum_{e \subset \partial E\cap\Gamma} \langle  u - \uint, \Pi^{0,e}_{k-1} v_h \rangle_{e} +
\sum_{e \subset \partial E \cap \Gamma_{\text{in}}}\langle | \bb \cdot \nn^{E} | ( u - \P0 \uint ), \P0 v_h\rangle_{e}\\
&= T_{N,1}^E + T_{N,2}^E+ T_{N,3}^E + T_{N,4}^E \, .
\end{aligned}
\end{equation}
For $T_{N,1}^E$, we have
\begin{equation}\label{eq:TN1E.0}
\begin{aligned}
T_{N,1}^E
& = 
\sum_{e \subset \partial E \cap \Gamma} \epsilon \langle   
\PZ0P \nabla  u_{\mathcal{I}} \cdot \nn^E-\nabla u \cdot \nn^E,
v_h \rangle_{e} \\
& \le 
 \sum_{e \subset \partial E \cap \Gamma} \epsilon\| \nabla u - \PZ0P \nabla  u_{\mathcal{I}}  \|_{0,e} \| v_h \|_{0,e}\\
& \lesssim 
\epsilon h^{-\frac{1}{2}} \left(   \| \nabla u -  \PZ0P \nabla  u_{\mathcal{I}}\|_{0,E} + h\, | \nabla u -  \PZ0P \nabla  u_{\mathcal{I}} |_{1,E}    \right)
\sum_{e \subset \partial E \cap \Gamma}\| v_h \|_{0,e} \, .
\end{aligned}
\end{equation}
Since it holds (cf. Lemma \ref{lm:bramble} and Lemma~\ref{lm:interpolation}): 
\begin{equation*}\label{eq:TN1E.1}
\begin{aligned}
\| \nabla u -  \PZ0P \nabla  u_{\mathcal{I}}\|_{0,E} & \leq 
\| \nabla u -  \PZ0P \nabla  u \|_{0,E} +
\| \PZ0P (\nabla  u -  \nabla  u_{\mathcal{I}})\|_{0,E}\\
& \lesssim h^\reg \vert u \vert_{\reg +1,E}\, ,
\end{aligned}
\end{equation*}
and 
\begin{equation*}\label{eq:TN1E.2}
\begin{aligned}
| \nabla u -  \PZ0P \nabla  u_{\mathcal{I}}|_{1,E} & \leq 
| \nabla u -  \PZ0P \nabla  u |_{1,E} +
| \PZ0P (\nabla  u -  \nabla  u_{\mathcal{I}})|_{1,E}\\
& \lesssim| \nabla u -  \PZ0P \nabla  u |_{1,E} +
h^{-1} \| \PZ0P (\nabla  u -  \nabla  u_{\mathcal{I}})\|_{0,E}\\
& \lesssim h^{\reg -1} \vert u \vert_{\reg +1,E} \, , 
\end{aligned}
\end{equation*}
we get
\begin{equation}\label{eq:TN1E}
\begin{aligned}
T_{N,1}^E
& \lesssim
\epsilon h^\reg \vert u \vert_{\reg +1,E} \, h^{-\frac{1}{2}}
\sum_{e \subset \partial E \cap \Gamma}\| v_h \|_{0,e} \\
&\lesssim  
\epsilon h^\reg \vert u \vert_{\reg +1,E} \, h^{-\frac{1}{2}}\sum_{e \subset \partial E \cap \Gamma}\bigl( \| v_h - \Pi^{0,e}_0 v_h \|_{0,e} + \| \Pi^{0,e}_0 v_h \|_{0,e}\bigr) \\
&\lesssim  
\epsilon h^\reg \vert u \vert_{\reg +1,E} \left( | v_h |_{1,E}
+  \sum_{e \subset \partial E \cap \Gamma} h^{-\frac{1}{2}} \| \Pi^{0,e}_{k-1} v_h \|_{0,e}\right)   \\
&\lesssim  
\epsilon^{\frac{1}{2}} h^\reg |u|_{\reg +1,E} \| v_h \|_{\cip,E}\, ,
\end{aligned}
\end{equation}
where, in the penultimate step, we have used the inequality established in the proof of Lemma~\ref{lm:bh-edge}.

For $T_{N,2}^E$, we get
\begin{equation}\label{eq:TN2E}
\begin{aligned}
T_{N,2}^E
& = \sum_{e \subset \partial E \cap \Gamma} \epsilon \langle \uint - u , \PZ0P \nabla  v_h \cdot \nn^{E} \rangle_{e} \\
& \le 
\sum_{e \subset \partial E \cap \Gamma} \epsilon \| u - \uint \|_{0,e}
\| \PZ0P \nabla  v_h  \|_{0,e} \\
& \lesssim
\sum_{e \subset \partial E \cap \Gamma} \epsilon \| u - \uint \|_{0,e}\, 
h^{-\frac{1}{2}}\| \PZ0P \nabla  v_h  \|_{0,E} \\
& \lesssim
\sum_{e \subset \partial E \cap \Gamma}\epsilon h^{-\frac{1}{2}} \| u - \uint \|_{0, e}
| v_h |_{1,E} \\
& \lesssim
\epsilon^{\frac{1}{2}} \left( | u - \uint |_{1,E} + h^{-1} \| u - \uint \|_{0,E} \right) \| v_h \|_{\cip,E} \\
& \lesssim
\epsilon^{\frac{1}{2}} h^\reg | u |_{\reg +1,E} \| v_h \|_{\cip,E} \, .
\end{aligned}
\end{equation}

For $T_{N,3}^E$, we have
\begin{equation}\label{eq:TN3E}
\begin{aligned}
T_{N,3}^E
&= 
\dfrac{\epsilon}{\delta h_E}  \sum_{e \subset \partial E\cap\Gamma} \langle  u - \uint, \Pi^{0,e}_{k-1} v_h \rangle_{e} \\
& \le 
\left(  \dfrac{\epsilon}{\delta h_E}  \sum_{e \subset \partial E\cap\Gamma} \|  u - \uint  \|_{0,e}^2  \right)^{\frac{1}{2}}  \left(  \dfrac{\epsilon}{\delta h_E}  \sum_{e \subset \partial E\cap\Gamma} \|  \Pi^{0,e}_{k-1} v_h \|_{0,e}^2  \right)^{\frac{1}{2}} \\
& \lesssim  \left(  \dfrac{\epsilon}{\delta h_E}  \sum_{e \subset \partial E\cap\Gamma} \|  u - \uint  \|_{0,e}^2  \right)^{\frac{1}{2}}  \| v_h \|_{\cip, E} \\
& \lesssim   \dfrac{\epsilon^{\frac{1}{2}}}{\delta^{\frac{1}{2}} }  \left(  ( |  u - \uint  |_{1,E}^2 + h^{-2}\|  u - \uint  \|_{0,E}^2 ) \right)^{\frac{1}{2}}   \| v_h \|_{\cip, E} \\
& \lesssim
\epsilon^{\frac{1}{2}} h^\reg | u |_{\reg +1,E} \| v_h \|_{\cip, E}.
\end{aligned}
\end{equation}

Finally, for $T_{N,4}^E$, we get
\begin{equation}\label{eq:TN4E}
\begin{aligned}
T_{N,4}^E
&= 
\sum_{e \subset \partial E \cap \Gamma_{\text{in}}} \langle | \bb \cdot \nn^{E} | ( u - \P0 \uint ), \P0 v_h\rangle_e \\
& \lesssim
\sum_{e \subset \partial E \cap \Gamma_{\text{in}}} \| u - \P0 \uint \|_{0,e} \, \| | \bb \cdot \nn^{E}|^{\frac{1}{2}} \P0 v_h \|_{0,e} \\
&\lesssim
\sum_{e \subset \partial E \cap \Gamma_{\text{in}}} \bigl( \| u - \P0 u \|_{0,e} + \| \P0 (u -  \uint ) \|_{e} \bigr) \, \| v_h \|_{\cip,E}\\
&\lesssim
\left( h^{-\frac{1}{2}}\| u - \P0 u \|_{0,E} + h^{\frac{1}{2}} | u - \P0 u |_{1,E} 
+ h^{-\frac{1}{2}} \| \P0 (u -  \uint )\|_{0,E} \right) \, \| v_h \|_{\cip,E}\\
&\lesssim
\left( h^{-\frac{1}{2}}\| u - \P0 u \|_{0,E} + h^{\frac{1}{2}} | u - \P0 u |_{1,E} 
+ h^{-\frac{1}{2}} \| u -  \uint \|_{0,E} \right) \, \| v_h \|_{\cip,E}\\
&\lesssim
h^{\reg + \frac{1}{2}} | u |_{\reg +1,E} \, \| v_h \|_{\cip,E} \, .
\end{aligned}
\end{equation}
Estimate \eqref{eq:errn} now follows by considering estimates \eqref{eq:TN1E}-\eqref{eq:TN4E} and summing all the local contributions.
\end{proof}

Combining Lemmas~\ref{lm:est_eI}--\ref{lm:errN} with Proposition \ref{prp:abstract}, we obtain the following result.

\begin{theorem}
\label{th:bfb}
Let $u$ be the solution of problem~\eqref{eq:weak-2} and $u_h \in V_h^{k,nc}(\Omega_h)$ be the solution of the discrete problem~\eqref{eq:cip-vem}. 
Under assumptions~\textbf{(A1)} and \textbf{(A2)},
it holds true that
\begin{equation}\label{eq:errest}
\begin{aligned}
\|u - u_h\|_{\cip} \lesssim
\left(
\epsilon^\frac{1}{2} \, h^{\reg} 
+   
h^{\reg+\frac{1}{2}} 
\right) \, 
\Big(\sum_{E \in \Omega_h} \Theta_E^2 \Big)^{1/2} ,
\end{aligned}
\end{equation}

with constants $\Theta_E$ depending on
$|u|_{\reg +1,E}$, $|f |_{\reg +\frac{1}{2},E}$, and $\|\beta\|_{[W^{\reg +1, \infty}(E)]^2}$, but independent of~$h$ and~$\varepsilon$.
\end{theorem}

\section{Numerical results}\label{s:numerical}

In this section we numerically test our method in two space dimensions, by considering two model problems in the domain $\Omega = [0,1] \times [0,1]$.
We use two different families of meshes:
\begin{itemize}
\item \texttt{octag:} meshes obtained by perturbing structured triangular meshes: each hypotenuse is split into two edges, then all nodes are perturbed, finally one extra node (the midpoint) is introduced on each edge; 
the elements of the obtained meshes are octagons;
\item \texttt{voro:} Voronoi meshes. 
\end{itemize}
\begin{figure}[!htb]
\begin{center}
\begin{tabular}{ccc}
\texttt{octag} & &\texttt{voro} \\
\includegraphics[width=0.35\textwidth]{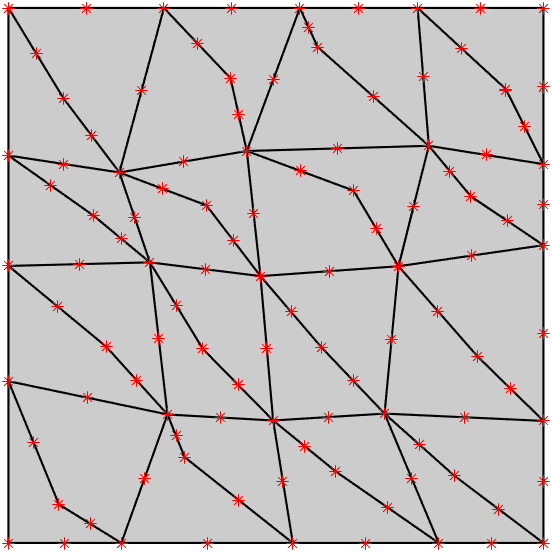} &\phantom{mm}&
\includegraphics[width=0.35\textwidth]{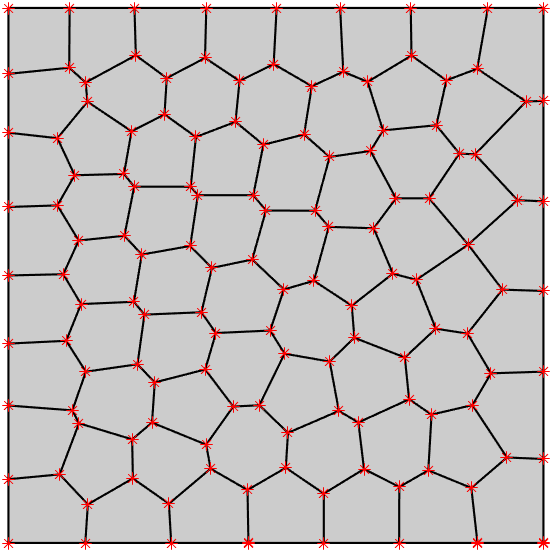} \\
\end{tabular}
\end{center}
\caption{Example of meshes used for the tests.}
\label{fig:meshes}
\end{figure}
The analytic expression of the VEM solution $u_h$ is unknown and we cannot compute the difference $u - u_h$ in closed form. Therefore, we consider the following quantities:
\begin{itemize}
\item \textbf{$H^1-$seminorm error}
$$
e_{H^1} := \sqrt{\sum_{E\in\Omega_h}\left\|\nabla(u-\PN u_h)\right\|^2_{0,E}}\,;
$$
\item \textbf{$L^2-$norm error}
$$
e_{L^2} := \sqrt{\sum_{E\in\Omega_h}\left\|(u-\P0 u_h)\right\|^2_{0,E}}\,.
$$
\end{itemize}
We also consider the error in the CIP-norm defined in \eqref{eq:glb_norm_def}:
$$
\begin{aligned}
\| u - u_h \|^2_{\cip} 
& \approx 
\sum_{E\in\Omega_h} \epsilon \left\|\nabla(u-\PN u_h)\right\|^2_{0,E}
+
h \sum_{E\in\Omega_h}\left\| \bb \cdot \nabla \P0 ( u - u_h) \right\|^2_{0,E}
\\
&+ \sum_{E\in\Omega_h} \sigma \left\|(u-\P0 u_h)\right\|^2_{0,E}
+
\dfrac{\epsilon}{\delta h}\sum_{e \in \mathcal{E}_h^\partial} \left\|  \Pi^{0,e}_{k-1} (u - u_h) \right\|^2_{0,e}
 \\
&+
\sum_{e \in \mathcal{E}_h^\partial,\ e\subset \Gamma_{\text{in}}} \left\| | \bb \cdot \nn^{E} |^{\frac12} \Pi^{0,E}_k (u - u_h) \right\|^2_{0,e}
+J(u - u_h, u - u_h).
\end{aligned}
$$

We assume that the analytic solutions of problem \eqref{eq:strong} are the functions
\[
u_1(x,y) = \dfrac{1}{2} \left(1 - \tanh \left(\dfrac{x - 0.5}{0.05} \right)\right) \qquad \mbox{(first test)}\, ,
\]
\[
u_2(x,y) = (y - y^2) \left( x - \dfrac{e^{\frac{x-1}{0.05}} - e^{\frac{-1}{0.05}}}{1 - e^{\frac{-1}{0.05}}} \right) \qquad \mbox{(second test)}\, .
\]
The first solution exhibit an internal layer in the middle of the domain. The second solution vanishes along the whole boundary and has a boundary layer at $x=1$.
For both tests, the parameter $\sigma$ is set to $1$, while the convective coefficient is 
\[
\bb(x,\,y) := \left[\begin{array}{r}
-2\,\pi\,\sin(\pi\,(x+2\,y))\\
\pi\,\sin(\pi\,(x+2\,y))
\end{array}\right]\,.
\]
The Nitsche parameter is selected as $\delta = 0.1$, while for the parameters $\gamma_e$ and $\gamma_E$, see \eqref{eq:gamma-loc}, we set $\kappa_e = \kappa_E = 0.025$.

\begin{remark}
The choice $\kappa_e = \kappa_E = 0.025$ appears in the numerical results of \cite{burman:2004}.
We also remark that in \cite{burman:2007} it has been shown that these parameters should be scaled as $k^{-\frac{7}{2}}$, where $k$ is the order of the method.
This highlights that the larger the discrete space, the less CIP stabilization is needed.
\end{remark}

\paragraph{Convergence rates of the schemes.}
We investigate the convergence rates of the approximation error, choosing $\epsilon = 10^{-5}$ (hence the problem is in the advection-dominated regime). We consider both the mesh families described above.

\begin{figure}
\begin{center}
\begin{tabular}{ccccc}
\includegraphics[width=0.43\textwidth]{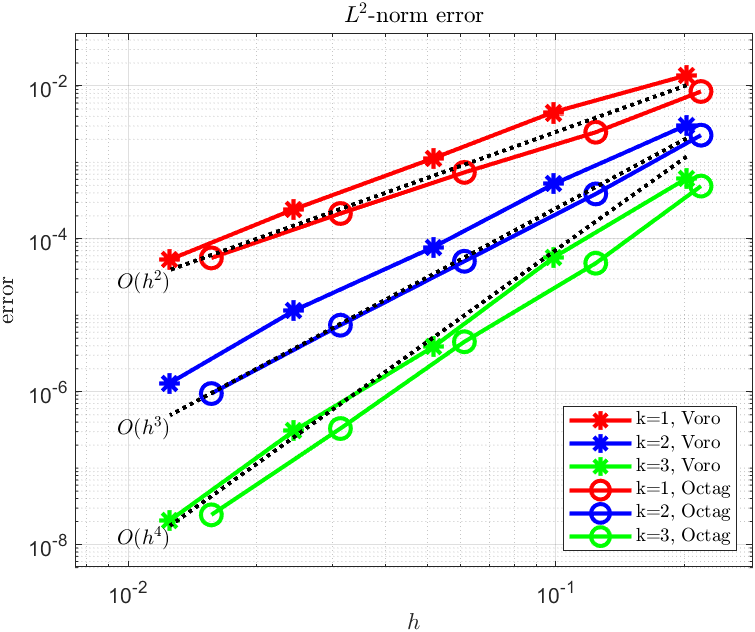}
&\phantom{mm}&
\includegraphics[width=0.43\textwidth]{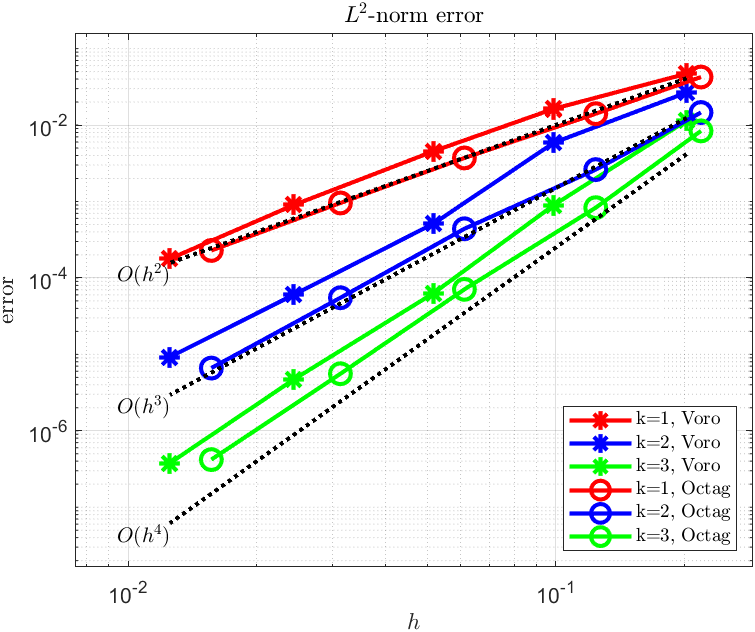} 
\end{tabular}
\begin{tabular}{ccccc}
\includegraphics[width=0.43\textwidth]{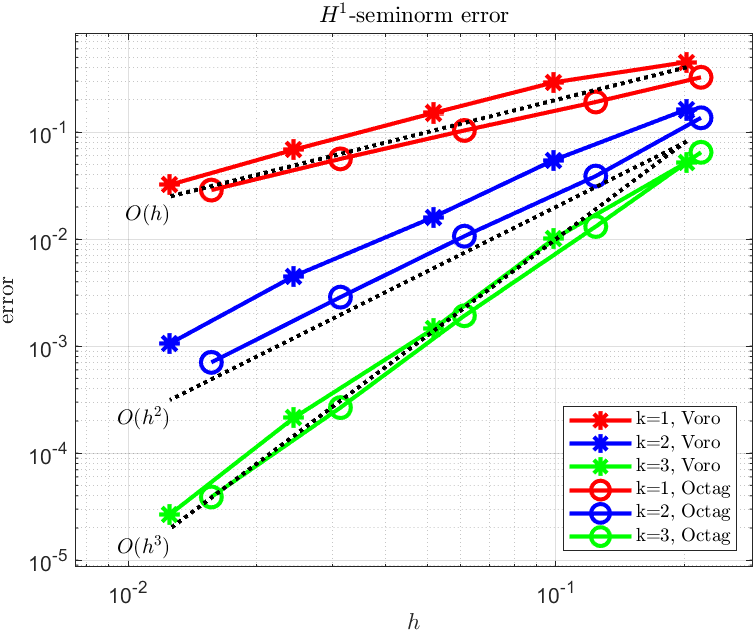}
&\phantom{mm}&
\includegraphics[width=0.43\textwidth]{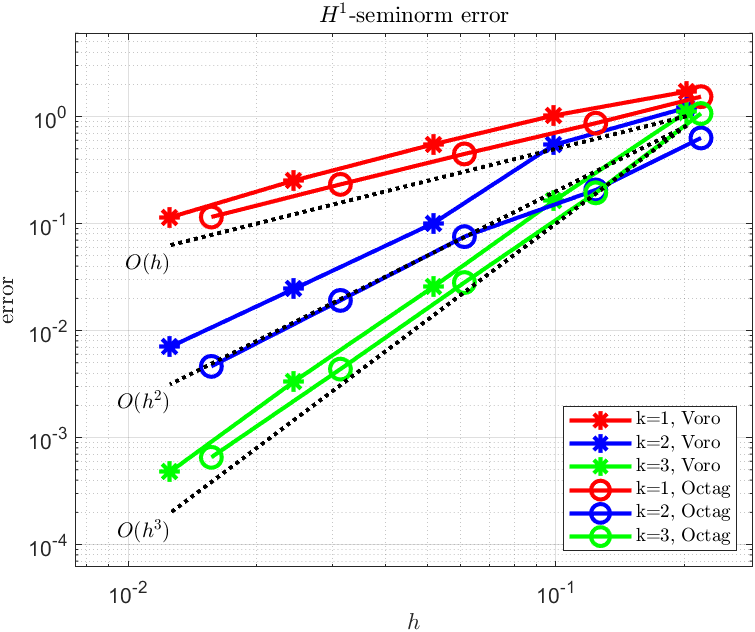} 
\end{tabular}
\begin{tabular}{ccccc}
\includegraphics[width=0.43\textwidth]{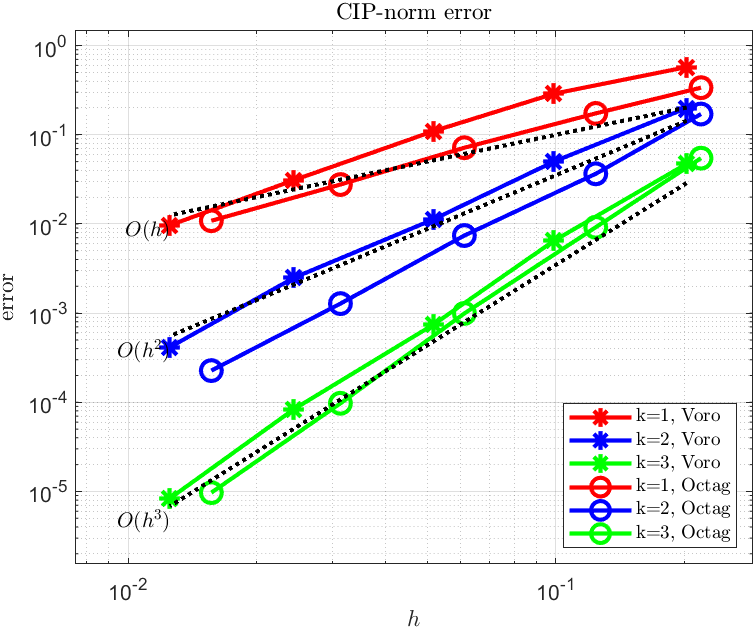}
&\phantom{mm}&
\includegraphics[width=0.43\textwidth]{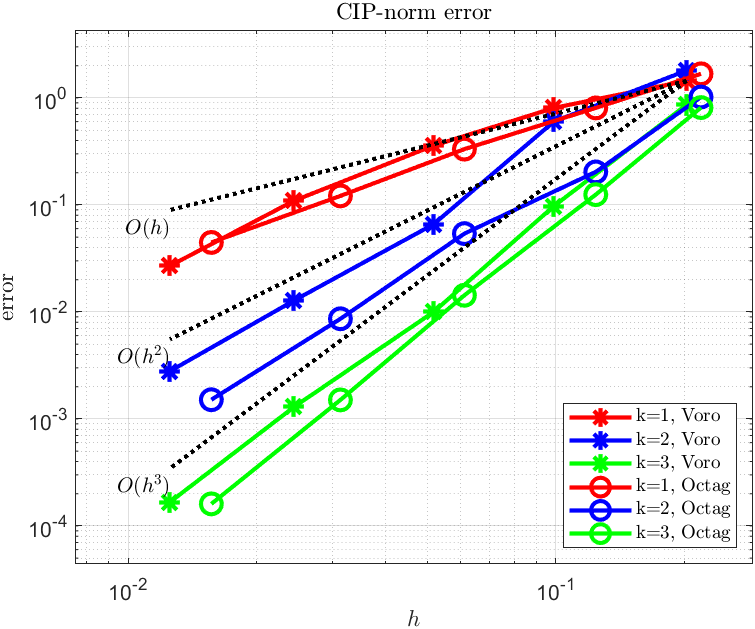} 
\end{tabular}
\end{center}
\caption{Convergences for $u_1$ (left column) and $u_2$(right column).} 
\label{fig:convergence}
\end{figure}

In Figure \ref{fig:convergence}, we observe the results for $u_1$ and $u_2$,
when the method orders are $k=1,2,3$. 
Optimal rates of convergence in the $L^2-$norm and $H^1-$seminorm can be appreciated. 
For the CIP-norm and $k=1$, we observe a super-linear convergence rate. This behavior can be explained by considering the error estimate of Theorem \ref{th:bfb}: most likely, for such errors and small $\epsilon$, the dominating part is the second one in the right-hand side of \eqref{eq:errest}, which correspond to a convergence rate of order $\frac{3}{2}$.

\paragraph{Robustness with respect to the parameter $\epsilon$.}
We now numerically assess the robustness of the method with respect to the advection parameter $\epsilon$. For this purpose, we test the method on a Voronoi tessellation with 1024 polygons.
\begin{figure}
\begin{center}
\begin{tabular}{ccccc}
\includegraphics[width=0.43\textwidth]{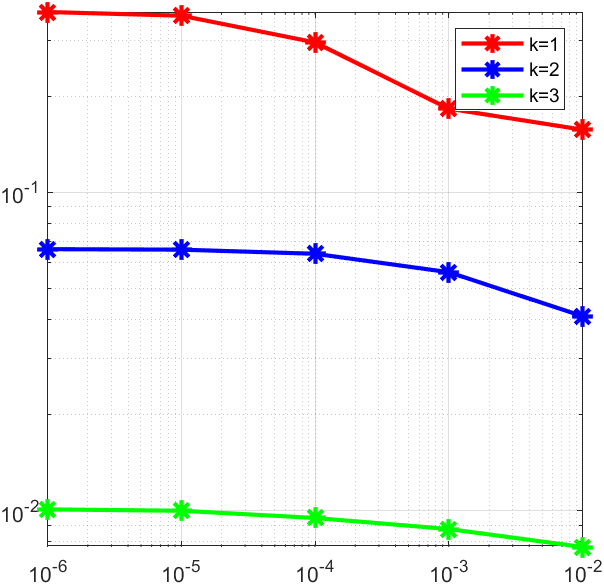}
&\phantom{mm}&
\includegraphics[width=0.43\textwidth]{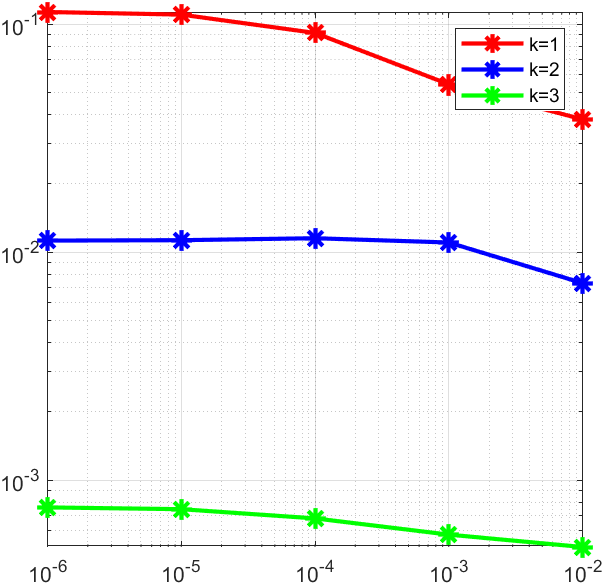} 
\end{tabular}
\end{center}
\caption{Results for various choice of $\epsilon$.} 
\label{fig:epsilon}
\end{figure}
As expected, in Figure \ref{fig:epsilon}, we observe that the CIP norm of the error is almost constant. Left column is related to $u_1$, right column is related to $u_2$. Similar result not reported here have been obtained using other meshes.

\section*{Funding}
\noindent
This research was funded in part by:\\
- the Austrian Science Fund (FWF) projects 10.55776/F65 and 10.55776/P33477 (IP);\\
- INdAM-GNCS (CL and MT), and the Italian MUR through the PRIN 2020 grant ``Advanced polyhedral discretisations of heterogeneous PDEs for multiphysics problems'' (CL).

\addcontentsline{toc}{section}{\refname}
\bibliographystyle{plain}
\bibliography{biblio}

\begin{thebibliography}{10}

\bibitem{AML-2016}
B.~Ayuso De~Dios, K.~Lipnikov, and G.~Manzini.
\newblock The nonconforming virtual element method.
\newblock {\em ESAIM Math. Model. Numer. Anal.}, 50(3):879--904, 2016.

\bibitem{volley}
L.~Beir\~{a}o~da Veiga, F.~Brezzi, A.~Cangiani, G.~Manzini, L.~D. Marini, and
  A.~Russo.
\newblock Basic principles of {V}irtual {E}lement {M}ethods.
\newblock {\em Math. Models Methods Appl. Sci.}, 23(1):199--214, 2013.

\bibitem{hitchhikers}
L.~Beir\~{a}o~da Veiga, F.~Brezzi, L.~D. Marini, and A.~Russo.
\newblock The hitchhiker's guide to the virtual element method.
\newblock {\em Mathematical Models and Methods in Applied Sciences},
  24(08):1541--1573, 2014.

\bibitem{BBMR-2016}
L.~Beir\~{a}o~da Veiga, F.~Brezzi, L.~D. Marini, and A.~Russo.
\newblock Virtual {E}lement {M}ethod for general second-order elliptic problems
  on polygonal meshes.
\newblock {\em Math. Models Methods Appl. Sci.}, 26(4):729--750, 2016.

\bibitem{BDLV:2021}
L.~Beir\~{a}o~da Veiga, F.~Dassi, C.~Lovadina, and G~Vacca.
\newblock {SUPG}-stabilized virtual elements for diffusion-convection problems:
  A robustness analysis.
\newblock {\em ESAIM Math. Model. Numer. Anal.}, 55(5):2233--2258, 2021.

\bibitem{BLR:2017}
L.~Beir\~{a}o~da Veiga, C.~Lovadina, and A.~Russo.
\newblock Stability analysis for the virtual element method.
\newblock {\em Math. Mod.and Meth. in Appl. Sci.}, 27(13):2557--2594, 2017.

\bibitem{BLT:2024}
L.~Beir\~{a}o~da Veiga, C.~Lovadina, and M.~Trezzi.
\newblock {CIP}-stabilized {V}irtual {E}lements for
  diffusion-convection-reaction problems.
\newblock {\em IMA J. Numer. Anal.}, 2024, to appear.

\bibitem{ACTA-VEM}
L.~Beirao~da Veiga, F.~Brezzi, L.D. Marini, and A.~Russo.
\newblock The virtual element method.
\newblock {\em ACTA Numerica}, 32:123--202, 2023.

\bibitem{berrone:2016}
M.~F. Benedetto, S.~Berrone, A.~Borio, S.~Pieraccini, and S.~Scialò.
\newblock Order preserving {SUPG} stabilization for the virtual element
  formulation of advection-diffusion problems.
\newblock {\em Comput. Methods Appl. Mech. Engrg.}, 293:18--40, 2016.

\bibitem{BERRONE2018}
S.~Berrone, A.~Borio, and G.~Manzini.
\newblock {SUPG} stabilization for the nonconforming virtual element method for
  advection–diffusion–reaction equations.
\newblock {\em Computer Methods in Applied Mechanics and Engineering},
  340:500--529, 2018.

\bibitem{brenner-scott:book}
S.~C. Brenner and L.~R. Scott.
\newblock {\em The Mathematical Theory of Finite Element Methods}, volume~15 of
  {\em Texts in Applied Mathematics}.
\newblock Springer, New York, third edition, 2008.

\bibitem{brenner-sung:2018}
S.~C. Brenner and L.Y. Sung.
\newblock Virtual element methods on meshes with small edges or faces.
\newblock {\em Math. Models Methods Appl. Sci.}, 28(7):1291--1336, 2018.

\bibitem{BrezziDG}
Franco Brezzi and Endre Süli.
\newblock {Discontinuous Galerkin methods for first-order hyperbolic problems}.
\newblock {\em Mathematical Models and Methods in Applied Sciences}, 14, 02
  2004.

\bibitem{brooks1982}
Alexander~N. Brooks and Thomas~J.R. Hughes.
\newblock {Streamline upwind/Petrov-Galerkin formulations for convection
  dominated flows with particular emphasis on the incompressible Navier-Stokes
  equations}.
\newblock {\em Computer Methods in Applied Mechanics and Engineering},
  32(1):199--259, 1982.

\bibitem{burman:2007}
E.~Burman and A.~Ern.
\newblock Continuous interior penalty hp-finite element methods for advection
  and advection-diffusion equations.
\newblock {\em Mathematics of Computation}, 76(259):1119--1140, 2007.

\bibitem{burman:2004}
E.~Burman and P.~Hansbo.
\newblock Edge stabilization for galerkin approximations of
  convection–diffusion–reaction problems.
\newblock {\em Computer Methods in Applied Mechanics and Engineering},
  193(15):1437--1453, 2004.

\bibitem{CMS-2018}
A.~Cangiani, G.~Manzini, and O.~Sutton.
\newblock Conforming and nonconforming virtual element methods for elliptic
  problems.
\newblock {\em IMA J. Numer. Anal.}, 37(3):1317--1354, 2017.

\bibitem{HHO-book-2}
Matteo Cicuttin, Alexandre Ern, and Nicolas Pignet.
\newblock {\em Hybrid high-order methods: a primer with applications to solid
  mechanics}.
\newblock Springer, 2021.

\bibitem{HHO-book-1}
Daniele~Antonio Di~Pietro and J{\'e}r{\^o}me Droniou.
\newblock The hybrid high-order method for polytopal meshes.
\newblock {\em Number 19 in Modeling, Simulation and Application}, 2020.

\bibitem{DE-book}
Daniele~Antonio Di~Pietro and Alexandre Ern.
\newblock {\em Mathematical aspects of discontinuous Galerkin methods},
  volume~69.
\newblock Springer Science \& Business Media, 2011.

\bibitem{douglas}
Jim Douglas and Todd Dupont.
\newblock Interior penalty procedures for elliptic and parabolic galerkin
  methods.
\newblock In R.~Glowinski and J.~L. Lions, editors, {\em Computing Methods in
  Applied Sciences}, pages 207--216, Berlin, Heidelberg, 1976. Springer Berlin
  Heidelberg.

\bibitem{Junt-Stenb}
M.~Juntunen and R.~Stenberg.
\newblock {N}itsche’s method for general boundary conditions.
\newblock {\em Mathematics of Computation}, 78:1353--1374, 2009.

\bibitem{LPS0}
Petr Knobloch and Gert Lube.
\newblock Local projection stabilization for advection--diffusion--reaction
  problems: One-level vs. two-level approach.
\newblock {\em Applied numerical mathematics}, 59(12):2891--2907, 2009.

\bibitem{li2021}
Yang Li and Minfu Feng.
\newblock A local projection stabilization virtual element method for
  convection-diffusion-reaction equation.
\newblock {\em Applied Mathematics and Computation}, 411:126536, 2021.

\bibitem{matthies2007unified}
Gunar Matthies, Piotr Skrzypacz, and Lutz Tobiska.
\newblock A unified convergence analysis for local projection stabilisations
  applied to the {O}seen problem.
\newblock {\em ESAIM: Mathematical Modelling and Numerical Analysis},
  41(4):713--742, 2007.

\bibitem{Nitsche}
J.~Nitsche.
\newblock {\"Uber ein Variationsprinzip zur L\"osung von {D}irichlet-Problemen
  bei Verwendung von Teilr\"aumen, die keinen Randbedingungen unterworfen
  sind}.
\newblock {\em Abh. Math. Sem. Univ. Hamburg}, 36:9--15, 1971.

\bibitem{sema-simai}
G.~Manzini P.F.~Antonietti, L. Beirao da~Veiga.
\newblock {\em The Virtual Element Method and its applications}, volume~31.
\newblock SEMA-SIMAI springer series, 2021.

\bibitem{R-book}
B{\'e}atrice Rivi{\`e}re.
\newblock {\em {Discontinuous Galerkin methods for solving elliptic and
  parabolic equations: theory and implementation}}.
\newblock SIAM, 2008.

\end{thebibliography}

\end{document}